\documentclass[11pt]{amsart}
\usepackage{amssymb}
\usepackage{amsfonts}
\usepackage{amscd}
\usepackage[T1]{fontenc}
\usepackage{color}
\usepackage{tikz}
\usepackage{hyperref}

\usetikzlibrary{matrix,arrows}

\long\def\private#1{}

\newbox\removebox
\newcommand\remove[1]{%
\setbox\removebox=\ifmmode\hbox{$#1$}\else\hbox{#1}\fi%
\leavevmode
\rlap{\textcolor{blue}{\vrule height0.8ex depth-0.6ex width\wd\removebox}}%
\box\removebox
}
\long\def\bigremove#1{%
\par\setbox\removebox=\vbox{#1}%
\vbox{%
\vbox to0pt{\hbox{\tikz\draw[color=blue,thick] (0,0) -- (\wd\removebox,-\ht\removebox)  (\wd\removebox,0) -- (0,-\ht\removebox);}}
\box\removebox
}
}

\usepackage{mathrsfs} 

\newcommand{\dcup}{\mathrel{\dot{\cup}}}

\newcommand{\inject}{\hookrightarrow}

\newcommand{\acl}{\operatorname{acl}}
\newcommand{\cl}{\operatorname{cl}}
\newcommand{\dcl}{\operatorname{dcl}}
\newcommand{\dimcl}{\dim^{\cl}}
\newcommand{\mdim}{\operatorname{mdim}_*}

\newcommand{\rk}{\operatorname{rk}}

\newcommand{\tp}{\operatorname{tp}}
\newcommand{\Vol}{\operatorname{Vol}}

\newcommand{\pow}{\mathcal{P}}

\newcommand{\Sym}{\hat{\mathcal{S}}}

\newcommand{\Ksb}[1][Z]{K^+_{b}(#1)}  
\newcommand{\Ks}[1][Z]{K^+(#1)}

\newcommand{\NNi}{\NN_{\infty}}

\newcommand{\Ni}{N_{\infty}}

\newcommand{\cNi}{\cN_{\infty}}

\newcommand{\refi}[2]{\ref{#1}~(\ref{#1.#2})}

\def\NN{{\mathbb N}}

\def\QQ{{\mathbb Q}}
\def\RR{{\mathbb R}}

\def\ZZ{{\mathbb Z}}

\def\cN{{\mathcal N}}

\def\cQ{{\mathcal Q}}

\def\cZ{{\mathcal Z}}

\newcommand{\tup}[1]{{\bar{#1}}}  
\newcommand{\atup}{\tup{a}}
\newcommand{\btup}{\tup{b}}
\newcommand{\ctup}{\tup{c}}
\newcommand{\ttup}{\tup{t}}
\newcommand{\xtup}{\tup{x}}
\newcommand{\ytup}{\tup{y}}

\newcommand{\csg}{\mathcal{C}\!\mathit{onv}(Z)}  
\newcommand{\dimspace}{\NN^{\csg}}

\newtheorem{thm}[subsubsection]{Theorem}
\newtheorem{lem}[subsubsection]{Lemma}
\newtheorem{lem-defn}[subsubsection]{Lemma-Definition}
\newtheorem{cor}[subsubsection]{Corollary}
\newtheorem{prop}[subsubsection]{Proposition}

\theoremstyle{definition}
\newtheorem{defn}[subsubsection]{Definition}
\newtheorem{exa}[subsubsection]{Example}
\newtheorem{rem}[subsubsection]{Remark}
\newtheorem{notn}[subsubsection]{Notation}

\theoremstyle{remark}

\theoremstyle{plain}

\numberwithin{equation}{subsection}

  {\par\medskip\noindent #1\par\begingroup%
    \advance\leftskip by 1em\advance\rightskip by 1em}%
  {\par\endgroup}

\newcommand{\1}{^{-1}}

\begin{document}

\setcounter{tocdepth}{1} 

\title[Definable sets up to definable bijections in $\ZZ$-groups]{Definable sets up to definable bijections in Presburger groups}

\author[R.~Cluckers]{Raf Cluckers}
\address{Universit\'e de Lille, Laboratoire Painlev\'e, CNRS - UMR 8524, Cit\'e Scientifique, 59655
Villeneuve d'Ascq Cedex, France, and,
KU Leuven, Department of Mathematics,
Celestijnenlaan 200B, B-3001 Leu\-ven, Bel\-gium}
\email{Raf.Cluckers@math.univ-lille1.fr}
\urladdr{http://rcluckers.perso.math.cnrs.fr/}

\author[I.~Halupczok]{Immanuel Halupczok}
\address{Mathematisches Institut, HHU D\"usseldorf,
Universit\"atsstr. 1, 40225 D\"usseldorf,
Germany}
\email{math@karimmi.de}
\urladdr{http://www.immi.karimmi.de/en/}



\begin{abstract}
We entirely classify definable sets up to definable bijections in $\ZZ$-groups, where the language is the one of ordered abelian groups.
From this, we deduce, among others, a classification of definable families of bounded definable sets.
\end{abstract}

\maketitle


\section{Introduction}

In \cite{Clu.cell}, the first author classified definable sets in the ordered abelian group $\ZZ$ up to definable bijection:
Finite sets are classified by their cardinality and infinite ones are classified by their dimension.
The present article has two main goals:
\begin{enumerate}
\item generalise this classification to $\ZZ$-groups (i.e., elementary extensions of $\ZZ$ in the language of ordered abelian groups; see Section~\ref{sec:Pres});
\item classify definable families of bounded sets
up to definable families of bijections (in $\ZZ$ and in elementary extensions).
\end{enumerate}

This may open the way to related questions on $p$-adic and motivic integrals, like a criterion for definable sets $X\subseteq \ZZ_p^{n}$ and $Y\subseteq \ZZ_p^m$ to have the same $p$-adic volume, refining results from \cite{Clu.DefQp} and \cite{HK.motInt}; see below for more details.

The classification of definable sets in a $\ZZ$-group $Z$ (Theorem~\ref{thm.groth}) is stated in the form of
an isomorphism $\Phi\colon \Sym\Ni \to \Ks$, where $\Ks$ is the Grothendieck semiring of definable sets in $Z$ (with parameters from $Z$) and where $\Sym\Ni$
is a semiring specified in terms of generators and relations (Definition~\ref{defn.semiring}). Roughly, the generators of $\Sym \Ni$ are the preimages under $\Phi$
of the intervals $[0, a) \subseteq Z$, where $a$ is either a positive element of $Z$ or $\infty$,
and the relations are:
\begin{itemize}
\item those which one obtains by gluing two intervals together;
\item the one coming from the definable bijection $[0, \infty) \to [0, \infty) \dcup [0, \infty)$;
\item those coming from bijections between finite sets.
\end{itemize}

The isomorphism $\Phi$ reduces the problem of determining whether two definable sets are in definable bijection to the problem of
checking whether two elements of $\Sym\Ni$ are equal, a problem which for general semirings is still highly non-trivial. However, in $\Sym\Ni$,
this is not too difficult: Proposition~\ref{prop.eq.test} explains how to check equality in $\Sym\Ni$.

From this classification, one then easily obtains some general results about definable sets: a Cantor--Schr\"oder--Bernstein Theorem (Corollary~\ref{cor.inin})
and several cancellation results (Corollaries~\ref{cor.cancel} and \ref{cor.cancel.mult}). We are not aware of more direct proofs of those results.

To prove injectivity of the map $\Phi$, we introduce invariants of definable sets which might be of independent interest.
As mentioned at the beginning, in $\ZZ$, a full set of invariants is given by cardinality (for finite sets) and dimension (as defined in \cite{Clu.cell}). In an elementary extension $Z \succ \ZZ$, we work with generalised versions of these two kinds of invariants: a notion of ``hyper-cardinality'', which entirely classifies
bounded definable sets and which behaves a lot like cardinality (see Definition~\ref{defn.count}, Proposition~\ref{prop.groth-b}), and a whole range of notions of dimensions, which differ in how long an interval $[0, a)$ has to be for it to
be considered as having dimension $1$ (Definitions~\ref{defn.dim.tp}, \ref{defn.dim}, \ref{defn.mdim}).

\medskip

Let us now consider Goal (2), and let us for simplicity first work in $\ZZ$.
It is known that for any definable family $(X_{\ytup})_{\ytup \in \ZZ^k}$ of finite sets,
the cardinality function $\ytup \mapsto \#X_{\ytup}$ is piecewise polynomial, i.e., there exists
a partition of the parameter space $\ZZ^k$ into finitely many definable pieces on each of which $\ytup \mapsto \#X_{\ytup}$
is a polynomial (see Section \ref{sec:Pres}). Given two families $(X_{\ytup})_{\ytup \in \ZZ^k}$ and $(X'_{\ytup})_{\ytup \in \ZZ^k}$,
the existence of a family of definable bijections $X_{\ytup} \to X'_{\ytup}$ clearly implies
that the cardinality functions are equal. Our main result about families (Theorem~\ref{thm.pres-fam}) states that this is an equivalence:
If $\#X_{\ytup} = \#X'_{\ytup}$ for every $\ytup$, then there exists a definable family of bijections.

This result can also be viewed as an ``automatic uniformity'' result for definable maps between finite sets:
The equality $\#X_{\ytup} = \#X'_{\ytup}$ just means that there exists a definable bijection $X_{\ytup} \to X'_{\ytup}$,
so Theorem~\ref{thm.pres-fam} states that the existence of individual definable bijections already implies the existence of
a uniformly definable family of bijections.
Note that the point here is that we only work in the structure $\ZZ$ itself. Otherwise (if we would assume
definable bijections $X_{\ytup} \to X_{\ytup'}$ to exist in elementary extensions of $\ZZ$), the
existence of a uniform family would follow directly by compactness.
Note also that this kind of automatic uniformity is false if one allows the sets $X_{\ytup}$, $X'_{\ytup}$ to be infinite (see Example~\ref{ex.unbd-fam}).

The same classification of definable families also works in elementary extensions $Z \succ \ZZ$, where instead
of requiring the $X_{\ytup}$ to be finite, we only require them to be bounded, and where we replace cardinality
by the above-mentioned hyper-cardinality: The function sending $\ytup$ to the hyper-cardinality of $X_{\ytup}$
is still piecewise polynomial in $\ytup$ (Proposition~\ref{prop.fam-poly}), and automatic uniformity
in families of bounded definable sets also holds in $Z$.

\medskip

In the context of motivic integration, \cite{HK.motInt}, Hrushovski and Kazhdan put the focus on describing the Grothendieck ring of definable sets in valued fields up to measure-preserving definable bijections. Pushing this focus further, they also
describe the Grothendieck ring of definable sets in the value group (up to certain definable bijections), similar to the results of this paper, but then for divisible ordered abelian groups \cite{HK.motHecke} instead of for $\ZZ$-groups. Understanding the Grothendieck ring of definable sets in $\ZZ$-groups may be a step in combining the focus of \cite{HK.motInt} with the one of \cite{CL.mot}, and to study definable sets in valued fields with $\ZZ$-group as value group, up to measure-preserving definable bijections.
In a way, our results are more precise than \cite{HK.motHecke} in the sense that we describe the Grothendieck semirings,
whereas most results in \cite{HK.motHecke} are about the Grothendieck rings (which moreover have been tensorized with $\QQ$).
As long as one restricts to bounded definable sets, this may not make a big difference: At least in $\ZZ$-groups $Z$,
the Grothendieck semiring $\Ksb$ of bounded definable sets maps injectively into the corresponding ring even when tensorized with $\QQ$
(this is Lemma~\ref{lem.inject}). However, the Grothendieck ring (as opposed to semiring) of all definable sets in a $\ZZ$-group $Z$ is trivial
because of the definable bijection between $Z$ and $Z \setminus \{0\}$. (In contrast, the corresponding ring in \cite{HK.motHecke}
is still non-trivial; see Theorem~3.12 of \emph{loc.\ cit.})

\medskip

Note that there are not yet many settings for which the Grothendieck semiring of $Z$-definable subsets in a structure $Z$ is explicitly known throughout all the models $Z$ of a theory; see \cite{KF.GrothOmin} for explicit Grothendieck rings in the context of o-minimal groups, and similarly, as mentioned just above, \cite{HK.motHecke}. It may be challenging to extend these studies to the semiring case. In view of \cite{Con.intPresb} it may also be interesting to study Grothendieck semi-rings for reducts of the Presburger language, throughout all models. In any o-minimal field (as opposed to group), a full classification of definable sets up to definable bijections is known, see \cite[Chapter 4, Remarks 2.14]{Dri.Omin}.

\medskip

The paper is organised as follows: In Section~\ref{sect.thm}, we fix our notation, define the ring $\Sym\Ni$ precisely and state the central result,
Theorem~\ref{thm.groth}. The next two sections are devoted to the proofs of surjectivity and injectivity of the map $\Phi\colon \Sym\Ni \to \Ks$, respectively. More precisely,
the main result of Section~\ref{sect.bij-to-cube} is Proposition~\ref{prop.bij-to-cube}, which states that every definable set is in definable bijection to a finite disjoint union of products of intervals.

In Section~\ref{sect.cor}, we collect some fruits of our work, in particular, the corollaries mentioned above and
the results about definable families.

\subsection{On Presburger groups}
\label{sec:Pres}
The study of $\ZZ$-groups (also called Presburger groups) was initiated by M.~Presburger \cite{Pre.comp}, who proved quantifier elimination in the so-called Presburger language, namely the language with symbols $+,-,0,1,<$ and for each integer $n>1$ a relation $\cdot \equiv \cdot \bmod n$ for congruence modulo $n$. The previously mentioned fact that cardinality functions $\ytup \mapsto \#X_{\ytup}$ for definable families $(X_{\ytup})_{\ytup \in \ZZ^k}$ of finite sets are piecewise polynomial follows from this quantifier elimination result (or, from the cell decomposition result from \cite{Clu.cell}), and is a special case of results in Section 4 of \cite{CL.mot}. It is hard to find a first historical reference of this fact, and it appears at least implicitly in work by J.-I.~Igusa and J.~Denef, but also in work by R.~P.~Stanley. We refer to \cite[Section~3.1]{Mar.modTh} and to Section \ref{sec:recall} for basic properties of $\ZZ$-groups.

\subsection{Acknowledgements}
R.C. was partially supported by the European Research Council under the European Community's Seventh Framework Programme (FP7/2007-2013) with ERC Grant Agreement nr. 615722
MOTMELSUM, and would like to thank the Labex CEMPI  (ANR-11-LABX-0007-01). I.H. was partially supported by the SFB~878 of the Deutsche Forschungsgemeinschaft.

\section{Precise statement of the results}
\label{sect.thm}

\subsection{Model theoretic terminology and conventions}
\label{sect.term}

In the entire article we use the language $L = (0, +, -, <)$ of ordered abelian groups.
Definable always means $L$-definable (with arbitrary parameters if not specified otherwise)
and all structures are $\ZZ$-groups, i.e., elementary extensions of $\ZZ$ with the natural $L$-structure.

Most of the time, $Z \succ \ZZ$ will denote the model we are interested in.
However, we will always work inside a monster model $\cZ \succ Z$: All definable sets are subsets of $\cZ^n$, and the fact that we are interested in $Z$ is reflected by imposing that our definable sets are $Z$-definable.
Our monster model does not need to be particularly monstruous: we only require $\cZ$ to be $\kappa^+$-saturated for some $\kappa \ge |Z|$.
Parameter sets will usually have to be ``small'', i.e., of cardinality less or equal to $\kappa$.

Tuples (which are denoted with a bar on top) are always of finite length.

By a type we always mean a complete type in this paper. The set of complete $n$-types over $Z$ is denoted by $S_n(Z)$. By a $\wedge$-definable or $\vee$-definable set, we
mean a set definable by an infinite conjunction or disjunction of formulas, respectively, with a finite number of free variables and where the total set of parameters has cardinality at most $\kappa$.

Given a (usually small) set $A \subseteq \cZ$, we write $\dcl(A)$ for its definable closure and $\acl(A)$ for its algebraic closure (in the sense of model theory).

Most of our results are formulated using the elementary substructure $Z \prec \cZ$ as a parameter set.
By the following lemma, allowing arbitrary parameter sets would not be more general.

\begin{lem}\label{lem.dcl}
For any subset $A \subseteq \cZ$, $\dcl(A)$ is an elementary substructure of $\cZ$.
\end{lem}

\begin{proof}
This follows e.g.\ from the results in \cite[Section~3.1]{Mar.modTh}: From the axiomatization of Presburger Arithmetic
given there, one easily deduces that $\dcl(A)$ is a model of Presburger Arithmetic. Now $\dcl(A) \prec \cZ$ follows from
the quantifier elimination result given there (due to \cite{Pre.comp}).
\end{proof}

\private{I checked the ``follows from Marker'' precisely. --Immi}

We use the convention $0 \in \NN$, and we set $\NNi := \NN \cup \{\infty\}$.

Given $Z$ and/or $\cZ$ as above, we define, in analogy to $\NN \subset \ZZ \subset \QQ$:
\begin{itemize}
 \item $N := \{a \in Z \mid a \ge 0\}$, $\Ni: = N \cup \{\infty\}$, $Q := Z \otimes_\ZZ \QQ$;
 \item $\cN := \{a \in \cZ \mid a \ge 0\}$, $\cNi := \cN \cup \{\infty\}$, $\cQ := \cZ \otimes_\ZZ \QQ$.
\end{itemize}

For $a, b \in \cZ$ with $a \le b$, we write $[a, b)$ for the interval $\{x \in \cZ \mid a \le x < b\}$. (For $a = b$, this yields the empty set.)
Note that even if we take $a, b \in Z$, $[a, b)$ denotes the interval in $\cZ$.

\begin{defn}
Say that a definable subset of $\cZ^n$ is \emph{bounded} if it is a subset of $[-a, a)^n$ for some $a \in \cN$.
\end{defn}

\subsection{Some semigroups and semirings}

In this paper, all semigroups are commutative, written additively, and have a $0$ element,
and all semirings are commutative and have a $0$ and a $1$. Homomorphisms respect these $0$ and $1$. We need the following variant of symmetric algebras.

\begin{defn}\label{defn.sym}
Let $G$ be a semigroup with a semigroup homomorphism $s \colon (\NN,+) \to G$. Then we define the ``reduced symmetric algebra'' $\Sym G$ over $G$ as the free commutative semiring with generators
$[g]$ for $g \in G$, modulo the following relations.
\begin{enumerate}
\item $[g] + [g'] = [g + g']$ for $g, g' \in G$
\item $[0] = 0$ (where in the left hand $0$ is the neutral element of $G$ and the right hand $0$ is the $0$ in $\Sym G$);
\item $[s(1)] = 1$ (where the right hand $1$ is the $1$ in $\Sym G$).
\end{enumerate}
\end{defn}

\private{
An evil example: $G = \ZZ \cup \{\infty\}$, where $a + \infty = \infty$ for any $a$. An easy computation
shows that $\Sym (G)$ is trivial: We have $\infty + (-1)\cdot\infty = 1\cdot\infty + (-1)\cdot \infty = (1 + -1)\cdot \infty = 0\cdot \infty = 0$
and hence $1 = 1 + 0 = 1 + \infty + (-1)\cdot\infty = (1 + \infty) + (-1)\cdot\infty = \infty + (-1)\cdot\infty = 0$

 In particular:\\
-- $G \to \Sym (G)$ is not injective\\
-- The map $\Sym (\NN) \to \Sym (G)$ (induced by $\NN \to G$) is not injective.

}

For the moment, we will often assume the map $s\colon \NN \to G$ to be implicitly given. Later, we will have a canonical $s$ anyway.

The first two relations ensure that the map $G \to (\Sym G, +)$ is a semigroup homomorphism and the last one ensures that the induced map $\NN \to \Sym G$ is a semiring homomorphism. (That $[s(a)]\cdot [s(b)] = [s(a\cdot b)]$ for $a, b \in \NN$ follows from the first relation and that natural numbers can be written as sums of $1$.) Since $\Sym G$ is generated by the image of $G$ and there are no additional relations, $\Sym G$ (together with the map $G \to \Sym G$) is the initial object
in the category of pairs $(R, f)$ where $R$ is a semiring and $f\colon G \to (R,+)$ is a semigroup homomorphism with the property that $f \circ s$ is a semiring homomorphism. Note also that $\Sym$ is a functor: A semigroup homomorphism $f\colon G \to G'$ compatible with the maps $\NN \to G$ and $\NN \to G'$ induces a semiring homomorphism $\Sym G \to \Sym G'$.

Definition~\ref{defn.sym} also works well when applied to groups:
\begin{lem}\label{lem.ZQ}
Suppose that $G$ is a group and $s\colon (\NN, +) \to G$ is a semigroup homomorphism.
Then $\Sym G$ is a ring, $s$ extends to a group homomorphism $\ZZ \to G$
and the composition $\ZZ \to G \to \Sym G$ is a ring homomorphism.
Moreover, $\Sym(G \otimes_{\ZZ}\QQ)$ is canonically isomorphic to $(\Sym G) \otimes_{\ZZ} \QQ$ and
the induced map $\QQ \to \Sym(G \otimes_{\ZZ}\QQ)$ is also a ring homomorphism.
\end{lem}

\begin{proof}
Easy; left to the reader.
\end{proof}

\private{I checked the proofs. --Immi}

In general, the canonical map $G \to \Sym G$ does not need to be injective.
However, in those cases we are interested in, it is. To prove this, we use the following two lemmas.

\begin{lem}\label{lem.SQ}
Suppose that $G$ is a torsion-free group and that (as before) $s\colon \NN  \to G$ is a semigroup homomorphism.
Then the map $G \to \Sym G$ is injective and $\Sym G$ is an integral domain.
\end{lem}

\begin{proof}
We may assume that $G$ is divisible. (Once we have the lemma for $G \otimes_\ZZ \QQ$, it follows for $G$.) Given $k \ge 1$, write $\mathcal S^k G$ for the $k$-th symmetric power of $G$, considered as a $\QQ$-vector space.
We identify $\mathcal S^k G$ with a sub-space of $\mathcal S^{k+1} G$ using the map
$g_1 \otimes \dots \otimes g_i \mapsto g_1 \otimes \dots \otimes g_i \otimes s(1)$.
Using this identification, we have
\[
\Sym G = \bigcup_{k \ge 1} \mathcal S^k G
\]
This description implies the lemma.
\end{proof}

\begin{lem}\label{lem.inject.general}
Suppose that $G \subseteq G'$ are two semigroups satisfying the following conditions for every $g, h \in G'$:
\begin{enumerate}
 \item If $g+h = 0$, then $g = h = 0$.
 \item If $g+h \in G$ then $g \in G$ and $h \in G$.
\end{enumerate}
Then the induced map $\Sym G \to \Sym G'$ is injective, and, identifying $\Sym G$ with its image in $\Sym G'$,
for any $g' \in G'$ whose image $[g'] \in \Sym G'$ lies in $\Sym G$, we have $g' \in G$.
\end{lem}

\begin{proof}
We may assume that $G\ne G'$.
We extend the semiring $\Sym G$ by a single element as follows: $S := \Sym G \cup \{\infty\}$, with $a + \infty = \infty$ for any $a \in S$,
$0 \cdot \infty = 0$ and $a \cdot \infty = \infty$ for any $a \ne 0$.
An easy computation shows that $S$ is a semiring again. In that computation, one uses assumption (1) to obtain distributivity:
For $a, b\in \Sym G$, we have: $(a + b)\cdot \infty = 0 \iff a + b = 0 \iff (a = 0 \wedge b = 0) \iff a\cdot \infty + b\cdot \infty = 0$.

Now we extend the map $G \to \Sym G$ to a map $f\colon G' \to S$ by sending any $g \in G' \setminus G$ to $\infty$.
Using assumption (2), we obtain that the images of this map satisfy all the relations of Definition~\ref{defn.sym}.
This implies that $f$ factors over $\Sym G'$, i.e., for $g' \in G'$, we have $f(g') = p([g'])$ for some surjective map
$p\colon \Sym G' \to S$. The composition $\Sym G \to \Sym G' \overset{p}{\to} S$ is just the inclusion map $\Sym G \to S$,
since $\Sym G$ is generated by elements of the form $[g]$ for $g \in G$, and such an element gets sent to
$f(g) = [g]$.

This implies that the map $\Sym G \to \Sym G'$ is injective.
Moreover, given $g' \in G'$, we have $g' \in G$ iff $f(g') \ne \infty$ iff $[g'] \in \Sym G$.
\end{proof}

We now apply this semiring construction to the semigroups we are really interested in: the (semi-)groups $N$, $\Ni$, $Z$ and $Q$ from Subsection~\ref{sect.term}.
Here, the map $s$ is always the embedding of $\NN$ coming from $\NN \subset \ZZ \prec Z$.

\begin{lem}\label{lem.inject}
In the following commutative diagram, all maps are injective:
\begin{center}
\begin{tikzpicture}
\matrix[matrix of math nodes,row sep=2ex,column sep={.5em}]{
      &|[name=ni1]| \Ni   &    &|[name=ni2]| \Sym \Ni \\
  |[name=n1]| N      &    &|[name=n2]| \Sym  N    \\
      &|[name=z1]| Z      &    &|[name=z2]| \Sym  Z  \\
   &   &|[name=q1]| Q      &    &|[name=q2]| \Sym  Q  \\
};
\draw[->] (n1) edge (ni1);  \draw[->]     (n2) edge (ni2);
\draw[->] (n1) edge (z1);  \draw[->]     (n2) edge (z2);
\draw[->] (z1) edge (q1);  \draw[->]     (z2) edge (q2);
\draw[->] (ni1) edge (ni2);
\draw[->] (n1) edge (n2);
\draw[->] (z1) edge (z2);
\draw[->] (q1) edge (q2);
\end{tikzpicture}
\end{center}
\end{lem}

\begin{proof}
Injectivity of $Q \to \Sym Q$ follows from Lemma~\ref{lem.SQ}.
Using injectivity of $N \inject Z \inject Q$, this implies
injectivity of all maps not involving the top line of the diagram.

By applying Lemma~\ref{lem.inject.general} to $N \subset N_\infty$
(which satisfies the assumptions of the lemma), we get that the map $\Sym N \to \Sym \Ni$ is injective,
and we moreover get that the image of $\infty \in \Ni$ in $\Sym \Ni$ lies outside of $\Sym N$,
which in particular implies (together with injectivity of $N \to \Sym N$) that
$\Ni \to \Sym \Ni$ is injective.
\end{proof}

\begin{notn}
From now one, we use Lemma~\ref{lem.inject} to identify $N$, $\Ni$, $Z$, and $Q$ with their images in $\Sym N$, $\Sym \Ni$, $\Sym Z$, and $\Sym Q$,
respectively. In particular, we omit the square brackets introduced in Definition~\ref{defn.sym}.
\end{notn}

\begin{rem}\label{rem.inj}
Given an elementary extension $Z \prec Z'$, we get that the map $\Sym Q \to \Sym Q'$ is injective (using the explicit construction of $\Sym Q'$ from the proof of Lemma~\ref{lem.SQ}). It now follows by Lemma \ref{lem.inject} that also the maps $\Sym N \to \Sym N'$, $\Sym Z \to \Sym Z'$ and $\Sym Q \to \Sym Q'$ are injective.
However, we do not see a simple proof that the map $\Sym \Ni \to \Sym \Ni'$ is
injective. However, it will follow \emph{a posteriori} once we know that those rings are isomorphic to certain Grothendieck rings
of definable sets; see the comments after Theorem~\ref{thm.groth},
\end{rem}

We will need some understanding of when two terms built out of elements of $N_\infty$ are equal as elements of $\Sym N_\infty$.

\begin{defn}\label{defn.eats}
For $a, b \in \Sym N_\infty$, we say that $a$ \emph{eats} $b$ if $a + b = a$.
\end{defn}

\begin{lem}\label{lem.eat}
Let $a_1, \dots, a_n$ and $b_1, \dots, b_n$ be elements of $N_\infty$.
Suppose that for each $i \le n$, there exists $k \in \NN$ such that $k a_i \ge b_i$.
Suppose moreover that $a_1 = \infty$. Then
$a := a_1\cdots a_n$ eats $b := b_1 \cdots b_n$ (in $\Sym N_\infty$).
\end{lem}

\begin{proof}
First, suppose that we have $a_i \ge b_i$ for each $i \le n$.
Set $a' := a_2 \cdots a_n$ and $b' := b_2 \cdots b_n$.
The assumption $a_i \ge b_i$ implies that there exists $c' \in \Sym \Ni$ with $a' = b' + c'$.
(Indeed, for $c_i := a_i - b_i \in \Ni$, we can set $c' = c_2(b_3\cdots b_n) + a_2c_3(b_4\cdots b_n) + \dots + (a_2\cdots a_{n-1})c_n$.)
Now we have
\[
a = \infty \cdot a' = \infty \cdot c' + \infty \cdot b'
= \infty \cdot c' + (\infty + b_1) \cdot b'
= \infty \cdot (c' + b') + b_1 \cdot b'
= a + b
.
\]
The general case now follows by applying the special case several times, as follows.
Let $a'_i$ be an integer multiple of $a_i$ with $a'_i \ge b_i$; then
$a'_1 \cdots a'_n = ka$ for some $k \in \NN_{\ge 0}$. By the previous part,
$\ell a$ eats $a$ for every $\ell \in \NN_{> 0}$, so
$(\ell + 1) a = \ell a$ and hence by induction $\ell a = a$ for all $\ell > 0$.
Moreover, $ka$ eats $b$ (also by the previous part) and hence
\[
a = ka = ka + b = a + b.\qedhere
\]
\end{proof}

With some more work, one can give a complete description of when $a$ eats $b$ for $a, b \in \Sym \Ni$. We postpone this to Corollary~\ref{cor.eat}
(where we will have more tools available).

\subsection{The main result about definable sets}

\begin{defn}\label{defn.semiring}
We write $\Ks$ for the Grothendieck semiring of $Z$-definable sets and
$\Ksb$ for the sub-semiring of bounded $Z$-definable sets. More precisely, $\Ks$ and $\Ksb$ are
generated, as abelian semigroups, by symbols $[X]$, where $X$ runs over the $Z$-definable (bounded, in the case of $\Ksb$) sets.
We have relations $[X_1 \cup X_2] = [X_1] + [X_2]$ if $X_1$ and $X_2$ are disjoint, and $[X_1] = [X_2]$ if there exists a
$Z$-definable bijection $X_1 \to X_2$. Multiplication is defined by $[X_1] \cdot [X_2] := [X_1 \times X_2]$.
\end{defn}

\begin{rem}\label{rem.bij}
The existence of definable Skolem functions implies that in Definition~\ref{defn.semiring}, it doesn't make a difference whether one requires the bijections
to be $Z$-definable or whether they can be definable over arbitrary parameters.
Indeed, if $\phi(\xtup, \ytup, \ctup)$ defines a bijection $X_1 \to X_2$ (for $X_1, X_2$ $Z$-definable and $\ctup$ arbitrary),
then using a definable Skolem function we find, in the ($Z$-definable) set of all parameters $\ctup'$ for which $\phi(\xtup, \ytup, \ctup')$ defines
a bijection $X_1 \to X_2$, one which lies in $\dcl(Z) = Z$.
\end{rem}

The map $\Ni \to \Ks, a \mapsto [[0, a)]$ is a semigroup homomorphism whose restriction to $\NN$ is a semiring homomorphism,
so by the remark after Definition~\ref{defn.sym}, we can define:

\begin{defn}\label{defn:Phi}
Let $\Phi\colon \Sym \Ni \to \Ks$ be the unique semiring homomorphism extending the map
$\Ni \to \Ks, a \mapsto [[0, a)]$.
\end{defn}

We can now precisely state the main result of this paper.

\begin{thm}\label{thm.groth}
The map $\Phi$ defined right above is a semiring isomorphism
$\Phi\colon \Sym \Ni \to \Ks$. Moreover, it
restricts to a semiring isomorphism $\Sym N \to \Ksb$.
\end{thm}

The semiring $\Ks[\ZZ]$ has already been determined in \cite{Clu.cell}: Definable sets up to bijection are classified by their dimension, and,
if the dimension is $0$, by their cardinality. Thus
$\Ks[\ZZ]$ is indeed isomorphic to $\Sym\NNi = \NN \cup \{\infty, \infty^2, \infty^3, \dots\}$.
A generalisation of this for bounded $Z$-definable sets using what we call hypercardinality is given in Section \ref{sec:hyper}.

Note that the map $\Phi$ is compatible with elementary extensions:
for $Z \prec Z'$ and $N' := \{a \in Z' \mid a\ge 0\}$, we have a commutative diagram
\[
\begin{array}{c@{\,\,}c@{\,\,}c}
\Sym \Ni  & \overset{\Phi}{\longrightarrow} & \Ks\\[.5ex]
\downarrow & & \downarrow \\[.5ex]
\Sym \Ni' & \smash{\overset{\Phi}{\longrightarrow}} & \Ks[Z']
\end{array}
\]
Since the vertical map on the right hand side is injective (by Remark~\ref{rem.bij}), we indirectly
also obtain that the natural map $\Sym \Ni \to \Sym \Ni'$ is injective.

\section{Surjectivity of $\Phi$}
\label{sect.bij-to-cube}

The main goal of this section is to prove that the map $\Phi$ from Definition~\ref{defn:Phi} 
is surjective, as part of Theorem~\ref{thm.groth}. Another way
to express this is that every definable set is in definable bijection to a disjoint union of products of intervals;
this is what Proposition~\ref{prop.bij-to-cube} states.

As before, we work with $\ZZ \prec Z \prec \cZ$, where $\cZ$ is $|Z|^+$-saturated, and definable sets live in $\cZ$.

\subsection{Recall: cell decomposition, rectilinearisation, piecewise linearity}\label{sec:recall}

In this subsection, we recall some results about $\ZZ$-groups, mainly from \cite{Clu.cell}.

\begin{defn}\label{defn.lin}
Given a definable set $X \subseteq \cZ^n$, we call a map $X \to \cZ$ \emph{linear} if it is of the form
\[
(x_1, \dots, x_n) \mapsto \frac1{a}(c + b_1x_1 + \dots + b_n x_n)
\]
for some $a \in \NN \setminus \{0\}$, $b_i \in \ZZ$ and $c \in \cZ$. We call a map $X \to \cZ \cup \{\infty\}$
linear if it is either a linear to $\cZ$ or constant equal to $\infty$.
We call a map $X \to \cZ^k$ linear if each of its components is linear.
\end{defn}

The central tool we use to understand definable sets in $Z$ is the Cell Decomposition Theorem from \cite{Clu.cell}.
Since we will need that result only ``up to linear bijection'' we can significantly simplify the statement
(avoiding to introduce the notion of cells).

\begin{lem}\label{lem.cell}
Every $Z$-definable set $X \subseteq \cZ^n$ can be partitioned into finitely many $Z$-definable pieces
in such a way that for each piece $X_i$, there exists a $Z$-definable linear bijection $X_i \to X'_i$ where
$X'_i$ is of the form
\[
X'_i = \{(y_1, \dots, y_n) \in \cN^n \mid y_1 \le \ell_1, y_2 \le \ell_2(y_1), \dots, y_n \le \ell_n(y_1, \dots, y_{n-1})\}
\]
for some linear $Z$-definable functions $\ell_i$ from the appropriate domain to $\cN \cup \{\infty\}$.
\end{lem}

\begin{proof}
First, we partition $X$ into cells using \cite[Theorem~1]{Clu.cell}. By refining the partition, we may assume
that each coordinate of each cell is bounded in at least one direction.
Then we get rid of the congruence conditions by scaling, and using another linear transformation, we
get to the desired form.
\end{proof}

\begin{rem}\label{rem.recti}
If $X$ is $\emptyset$-definable, one can do even better, namely one can obtain that each $X'_i$ in Lemma~\ref{lem.cell} is of the form
$\cN^k \times \{0\}^{n - k}$. This is \cite[Theorem~2]{Clu.cell} (``rectilinearisation''). Note that the parametric rectilinearisation, \cite[Theorem~3]{Clu.cell}, yields also a bit more for $Z$-definable $X$ than Lemma \ref{lem.cell}, but we will not use this in this paper.
\end{rem}

It follows from Lemma \ref{lem.cell}, applied to the graph of $f$, that any definable function is piecewise linear:

\begin{cor}\label{lem.lin}
For every $Z$-definable function $f\colon \cZ^n \to \cZ \cup \{\infty\}$, there exists a partition of $\cZ^n$ into finitely
many $Z$-definable sets such that for each part $X$, the restriction of $f$ to $X$ is linear.
\end{cor}

\subsection{Recall: dimension}

In $\cZ \succ \ZZ$, there are several different notions of dimension. Later in this paper, we will need a whole range of such notions,
but for the moment, we only need one of them, which has the property that the zero-dimensional sets are exactly the finite ones. There are various ways to define
it, e.g., using that $\acl$ has the exchange property and hence yields a notion of rank; cf.\ \cite[Definition~3]{Clu.cell}.

Instead of giving a precise definition here, we refer to Subsection~\ref{sect.dim}. Readers not familiar with dimension can read that subsection
up to Lemma~\ref{lem.dim-bas}; this is independent of the rest of the paper.

In the following, recall that by a $\wedge$-definable set, we mean an intersection of possibly infinitely many definable sets, but all of which
are definable with parameters from a single small parameter set.

\begin{defn}[{\cite[Definition~3]{Clu.cell}}]
We define the dimension of a $\wedge$-definable set $X$ by $\dim(X) := \dim^{\acl}(X)$, where
$\dim^{\acl}$ is introduced in Definition~\ref{defn.dim.tp}.
\end{defn}

See Remark~\ref{rem.dim-bas} and Lemma~\ref{lem.dim-bas} for some of basic properties of dimension. From those, we now deduce some more specific properties
that are needed below.

\begin{lem}\label{lem.fib1}
Let $f\colon X \to T$ be a $Z$-definable map (with $X$ and $T$ also $Z$-definable). Then the set $T_0 := \{\ttup \in T \mid \dim f^{-1}(\ttup) = \dim X\}$ is finite.
\end{lem}

\begin{proof}
By Lemma~\refi{lem.dim-bas}{def} ($\wedge$-definability of dimension), $T_0$ is $Z$-$\wedge$-definable (so that $\dim T_0$ makes sense).
If $T_0$ is infinite, then by saturatedness of $\cZ$, it contains an element not in
$\acl(Z)$. This implies $\dim T_0 \ge 1$ and hence
$\dim f^{-1}(T_0) = \dim X + 1$ (by Lemma~\refi{lem.dim-bas}{surj}), which is a contradiction.
\end{proof}

\begin{lem}\label{lem.fibp}
Let $(X_{\ttup})_{\ttup \in T}$ and $(Y_{\ttup})_{\ttup \in T}$ be $Z$-definable families of sets (parametrized by the same $Z$-definable set $T$),
let $d \in \NN$ be given,
and suppose that for every $\ttup \in T$ we have $\dim (X_{\ttup} \times Y_{\ttup}) \le d$.
Then there exists a partition of $T$ into finitely many $Z$-definable sets $T_i$
such that $\max\{ \dim X_{\ttup} \mid \ttup \in T_i \} +\max\{ \dim Y_{\ttup} \mid \ttup \in T_i \} \le d$ for every $i$.
\end{lem}

\begin{proof}
For each $i \le d$, the set
\[
T'_{i} := \{\ttup \in T \mid \dim X_{\ttup} \le i \wedge \dim Y_{\ttup} \le d - i\}
\]
is $Z$-$\vee$-definable by Lemma~\refi{lem.dim-bas}{def}. The union of those sets is all of $T$,
so by compactness (or, more precisely, saturatedness of $\cZ$), there are definable subsets $T_{i} \subseteq T'_{i}$ whose union still is all of $T$.
\end{proof}

\subsection{The world consists of cuboids (up to definable bijection)}

\begin{defn}
A \emph{cuboid} is a set of the form
$[0,a_1) \times \dots \times [0,a_k)$ for some $k \in \NN$ and some $a_1, \dots, a_k \in \cNi$.
\end{defn}

We will often consider ``disjoint unions of cuboids'' as definable sets. As usual, this means that we somehow
embed the cuboids into $\cZ^n$ for some $n$ in a disjoint way. This is harmless, since in reality,
we are interested in such disjoint unions only up to definable bijection.

Note that we consider the empty set as a cuboid. This feels a bit unnatural, but it will come in handy
in Corollary~\ref{cor.bij-to-cube} below.

\private{Namely when considering families: not every fiber consists of the same number of cuboids.}

Here is the main goal of this section:

\begin{prop}\label{prop.bij-to-cube}
Every $Z$-definable set $X$ is in $Z$-definable bijection to a finite disjoint union of $Z$-definable cuboids.
\end{prop}

Using a standard compactness argument, we also obtain a family version of this result, namely:

\begin{cor}\label{cor.bij-to-cube}
Let $Y$ be a $Z$-definable set and let $(X_{\ytup})_{\ytup \in Y}$ be a $Z$-definable family of sets parametrized by
$Y$. Then there exists $Z$-definable family of bijections $X_{\ytup} \to C_{\ytup}$, where
$(C_{\ytup})_{\ytup \in Y}$ is a finite disjoint union of $Z$-definable families of cuboids.
\end{cor}

Note that this does not only state that each $C_{\ytup}$ is a finite union of cuboids, but also that the number
of cuboids does not depend on $\ytup$. (This is where we need the empty set to be considered as a cuboid.)

\begin{proof}[Proof of Corollary~\ref{cor.bij-to-cube}]
Fix $\btup \in Y$ and set $Z' := \dcl(Z \cup \btup)$. By Lemma~\ref{lem.dcl}, we have $Z' \prec \cZ$, so
we can apply Proposition~\ref{prop.bij-to-cube} to obtain that there exists a $Z'$-definable bijection from $X_{\btup}$ to a disjoint union of $\ell_{\btup}$ $Z'$-definable cuboids.
Let $\phi_{\btup}(\xtup, \xtup', \btup) \in L(Z')$ define such a bijection (where parameters from $Z$ are omitted from the notation)
and let $\psi_{\btup}(\ytup) \in L(Z)$ state that
$\phi_{\btup}(\xtup, \xtup', \ytup)$ defines a bijection from $X_{\ytup}$ to a disjoint union of $\ell_{\btup}$ cuboids.

Since the (partial) type $\{\neg \psi_{\btup}(\ytup) \mid \btup \in Y\}$ is not realized in $Y$, by saturatedness of $\cZ$ it is inconsistent,
i.e., there exist finitely many $\btup_1, \dots \btup_m$ such that $\ytup \in Y$ implies $\bigvee_{i=1}^m\psi_{\btup_i}(\ytup)$.
This means that our desired family of bijections can be constructed using $\phi_{\btup_1},\dots \phi_{\btup_m}$.
\end{proof}

This corollary will be used inductively in the proof of Proposition~\ref{prop.bij-to-cube}.
For the induction to work, note that if we only know the proposition for definable sets $X$ of dimension at most $d$, then we still can deduce the corollary for families $(X_{\ytup})_{\ytup}$ whose fibers are all of dimension at most $d$.

\begin{proof}[Proof of Proposition~\ref{prop.bij-to-cube}]

If we have a partition of $X$ into finitely many definable sets, we may
treat each part separately; we will do this
several times during the proof.

We suppose $X \subseteq \cZ^n$ and we do a double induction: a main induction over $d := \dim X$, and for fixed $d$, an additional induction over $n$. In other words, we assume that
the statement is true for all sets $X'$ with $\dim X' < d$ and moreover for all sets with $\dim X' = d$ and $X' \subseteq \cZ^{n - 1}$.

Let $T$ be the projection of $X$ to the first coordinate and for $t \in T$ set $X_{t} = \{\xtup \in \cZ^{n-1} \mid (t, \xtup) \in X\}$.
By Lemma~\ref{lem.fib1}, there are only finitely many $t \in T$ for which $X_{t}$ has dimension $d$. For each of those $t$, we treat $X_{t}$ separately
by projecting it to the last $n-1$ coordinates and using induction on $n$ (and replacing $Z$ by $\dcl(Z, t) \prec \cZ$).
In this way, we may assume that $\dim X_{t} < d$ for every $t$.
In particular, by induction over $d$, using Corollary~\ref{cor.bij-to-cube} and after an additional partition of $X$,
we may assume that for each $t$, $X_{t}$ is a cuboid, i.e.,
$X_{t} = [0,\ell_1(t)) \times \dots \times [0,\ell_{k}(t))$ for some definable functions $\ell_i\colon T \to \cNi$.

We can also assume that $k = d - 1$. Indeed, using Lemma~\ref{lem.fibp} repeatedly, we may assume (after partitioning $T$ and partitioning $X$ accordingly)
that $\dim [0, \ell_i(t))$ is constant on $T$ for each $i$.
If $\dim [0, \ell_i(t)) = 0$ for every $t$, then
$\ell_i$ is bounded by a natural number and we can get rid of the factor $[0, \ell_i(t))$ by partitioning $X$ once more (according to the $i$-th coordinate).
Thus without loss $\dim [0, \ell_i(t)) = 1$ for all $i$ and hence $k = d - 1$. (We cannot have $k < d - 1$, since this would imply $\dim X < d$.)

By further partitioning $T$ (using Corolary~\ref{lem.lin}), we may assume that each $\ell_i$ is linear.
Moreover, using Lemma~\ref{lem.cell}, we can assume that $T$ is of the form $[0, s)$ for some $s \in N$.

We now describe the maps $\ell_i$ more precisely, so fix $i \in \{1, \dots, d - 1\}$.
Suppose first that $\ell_i(t)$ does not depend on $t$ (which is the case in particular if $\ell_i(t) = \infty$).
Then $X$ can be written as a product $[0, \ell_i(0)) \times X'$, with $X'$ being the projection of $X$ to all coordinates except the $i$-th one, and we are done by applying induction to $X'$. So from now on we suppose that $\ell_i$ is non-constant, i.e., that it is of the form $\ell_i(t) = \frac{1}{a_i}(c_i + b_it)$ for some integers $a_i > 0$ and $b_i \ne 0$ and some $c_i \in Z$.

Using
that $[0,s)$ is an infinite interval (otherwise we would have $\dim X = d - 1$),
we first deduce that $a_i$ divides $b_i$ and then that it divides $c_i$. In other words, without loss $a_i = 1$. Next, we do a case distinction on the sign of $b_i$.

Suppose first that $b_i > 0$. Then the interval $[0, c_i + b_it)$
is in definable bijection to the disjoint union of $[0, c_i)$ and $b_i$ copies of $[0,t)$
(note that $c_i \ge 0$ since $c_i + b_it \ge 0$ for $t = 0$).
We decompose $X$ according to this interval decomposition. The part corresponding to
$[0, c_i)$ is independent of $t$ so we already saw how to treat it and we are left with the case $\ell_i(t) = t$.

Now suppose that $b_i < 0$. Note that this cannot happen if $s = \infty$ since then $\ell_i(t)$ would be negative for $t$ sufficiently big.
This means that we can apply the same argument as for $b_i >0$, but with $t$ replaced by
$s - 1 - t$ (and using that $c_i + b_it \ge 0$ for $t = s - 1$) to reduce to
the case where $\ell_i(t) = s - 1 - t$.

To summarise, after a permutation of coordinates, we may assume that
$X_{t} = [0,t)^{r} \times [0, s - 1 - t)^{r'}$ for some $r,r' \in \NN$ satisfying $r + r' = d-1$.
Using another definable bijection (a translation of the last $r'$ coordinates), we change this to
$X_{t} = [0,t)^{r} \times [t+1, s)^{r'}$.

Now we rename $t =: x_1$ and denote the remaining coordinates by
$x_2, \dots, x_{d}$; in this notation, we have
\begin{align*}
\tag{*}
X = \{(x_1, \dots, x_{d}) \in [0, s)^d \mid
\,&x_2 < x_1, \dots, x_{r+1} < x_1, \\
&x_1 < x_{r+2}, \dots x_1 < x_{d}
\};
\end{align*}
it remains to show that a set of this form is in definable bijection to a finite union of cuboids.

We first treat the case $s = \infty$. In that case, we have $r = d - 1$ (since this implies that in the case distinction on the sign of $b_i$, $b_i < 0$ does not occur), and the map $(x_1, \dots, x_{d}) \mapsto (x_1 - \max\{x_2, \dots, x_{d}\} - 1, x_2, x_3, \dots, x_{d})$ sends
$X$ bijectively to the cuboid $[0, \infty)^d$. So now suppose $s < \infty$.
The remainder of the proof consist in cutting this (rather specific) set $X$ into pieces
and reassembling them differently to get cuboids.

For any permutation $\sigma \in S_{d}$, define the following set:
\[
X_{\sigma} := \{(x_1, \dots, x_{d}) \in [0, s)^d \mid x_{\sigma(1)} < \dots < x_{\sigma(d)}\}
.
\]
Moreover, set
\[
X_{0} := \{(x_1, \dots, x_{d}) \in X \mid \exists i, j \colon (i \ne j \wedge x_i = x_j) \}
.
\]
Our set $X$ is the disjoint union of a subset of $X_0$ and some of the sets $X_{\sigma}$ (namely those
$X_\sigma$ for which with $\sigma(1) > \sigma(i)$ for $i = 2, \dots, r+1$ and
$\sigma(1) < \sigma(i)$ for $i = r+2, \dots, d$).
The set $X_{0}$ has lower dimension, so we can treat $X\cap X_0$ by induction, and
we are left with sets of the form $X_{\sigma}$. In other words and after permuting
coordinates once again, we can assume that
\[
\tag{**}
X = \{(x_1, \dots, x_{d}) \in [0, s)^d \mid x_{1} < \dots < x_{d}\}
.
\]

The idea of the remainder of the proof is to cut this ``pyramid'' into many smaller
pyramids which then can be glued together to finitely many cuboids. 
To this end
we may without loss assume that $s$ is divisible by $d!$; indeed,
we can simply replace $s$ by $d!\cdot \lfloor \frac{s}{d!} \rfloor$
and treat the remainder, which has lower dimension, by induction. Set
$s' := \frac{s}{d!}$.
Now we further decompose $X$ as follows:
For $\atup \in [0, d!)^d$, define
\[
X_{\atup} := \{(x_1, \dots, x_{d}) \in [0, s') \mid
\underbrace{(x_1 + a_1s', \dots, x_{d} + a_{d}s') \in X}_{(\Delta)}
\}
.
\]
We have a definable bijection between $X$ and the disjoint union of all sets
$X_{\atup}$, where $(x_1,\dots, x_{d}) \in X_{\atup}$ is sent to
$(x_1 + a_1s', \dots, x_{d} + a_{d}s') \in X$.

Now fix $\atup\in [0, d!)^d$ and consider the set $X_{\atup}$. The condition $(\Delta)$ is the
conjunction over all $i = 1,\dots, d-1$ of the inequalities $x_i + a_is' < x_{i+1} + a_{i+1}s'$. If $a_i < a_{i+1}$, then this inequality is always true
(since $x_i, x_{i+1} \in [0, \dots, s')$); if $a_i > a_{i+1}$, then this
inequality is never true (and $X_{\atup}$ is empty); if $a_i = a_{i+1}$ then the inequality
is equivalent to $x_i < x_{i+1}$. In other words, each set $X_{\atup}$
is given by inequalities between some of the coordinates.

Now we further decompose each $X_{\atup}$ in the same way as we decomposed (*):
We treat the subset of those points where some coordinates are equal using induction,
and we decompose the remainder into subsets which all are, up to permutation of coordinates, of the form
\[
X' = \{(x_1, \dots, x_{d}) \in [0, s')^d \mid x_{1} < \dots < x_{d}\}
.
\]
Thus, to summarize, our definable set is now a disjoint union of a certain number of copies of $X'$.
An easy way to determine the number $u$ of copies we obtained from our original set (**) is to do an analoguous
decomposition in $\RR^d$ and to compare the volumes: if $Z := \{(z_1, \dots, z_{d}) \in \RR^d
\mid 0 \le z_1\le \dots \le z_{d} \le 1\}$ and $Z' := \{(z_1, \dots, z_{d}) \in \RR^d
\mid 0 \le z_1\le\dots \le z_{d} \le \frac1{d!}\}$, then $u = \Vol(Z)/\Vol(Z') = d!^d$;
in particular, $u$ is divisible by $d!$ (which is all we need to finish the proof).

Now the last step is to group $d!$ of these sets $X'$ into a cube (which in total
yields $d!^{d-1}$ cubes); the only technical difficulty here is to get back points where
not all coordinates are different; we achieve this by making the cubes slightly smaller.
More precisely, translate $X'$ by $(0, -1, \dots, -(d - 1))$, i.e., without loss
\[
X' = \{(x_1, \dots, x_{d}) \mid 0 \le x_{1} \le \dots \le x_{d} < s' - d + 1\}
.
\]
Next, we apply a different permutation of coordinates $\sigma \in S_{d}$ to each of our $d!$ many
sets $X'$; this yields sets
\[
X'_{\sigma} = \{(x_1, \dots, x_{d}) \mid 0 \le x_{\sigma(1)} \le \dots \le x_{\sigma(d)} < s' - d + 1\}
.
\]
The cube $C := [0, s' - d + 1)^d$ is the union of all sets $X'_{\sigma}$,
so to finish the proof, it remains to check that the intersection of any two sets
$X'_{\sigma_1}$ and $X'_{\sigma_2}$ (for $\sigma_1 \ne \sigma_2$) has
dimension less than $d$. (Then one copy of the intersection can be treated by induction.)

To see that $\dim (X'_{\sigma_1} \cap X'_{\sigma_2}) < d$, note that
there exist $i, i' \le d$, $i \ne i'$ such that any $\xtup \in X'_{\sigma_1}$
satisfies $x_i \le x_{i'}$, whereas any $\xtup \in X'_{\sigma_2}$
satisfies $x_i \ge x_{i'}$; hence any $\xtup$ in the intersection satisfies
$x_i = x_{i'}$.

\end{proof}

\section{Injectivity of $\Phi$}
\label{sect.inj}

\subsection{An invariant: hyper-cardinality}\label{sec:hyper}

To prove injectivity of the map $\Phi\colon \Sym \Ni \to \Ks$ from Theorem~\ref{thm.groth}, we need invariants of definable sets.
The first one we introduce is a generalization of cardinality; it will turn out to be already a full invariant for bounded definable sets.
To be able to define it, we need some preliminary lemmas.

In the following, we evaluate polynomials with coefficients in $\QQ$ at elements of $\cZ$,
by embedding both $\QQ$ and $\cZ$ into the ring $\Sym \cQ = \Sym \cZ \otimes_{\ZZ} \QQ$. (In particular, the result of such an evaluation lies in $\Sym \cQ$.)

\begin{lem}\label{lem.vanish}
Suppose that $Y \subseteq \cZ^k$ is a $Z$-definable set and that $f \in \Sym \cQ[y_1, \dots, y_k]$
is a polynomial which vanishes on $Y \cap Z^k$. Then $f$ vanishes on all of $Y$.
\end{lem}

\begin{proof}
We first assume that $Y$ is $\emptyset$-definable.
To prove the lemma, we can treat each part of a finite partition of $Y$ separately, and we can also apply
a ($\emptyset$-definable) linear transformation to $Y$ (provided that we apply the corresponding transformation to $f$).
By Rectilinearisation (\cite[Theorem~2]{Clu.cell}; cf.\ Remark~\ref{rem.recti}), this allows us
to reduce to the case $Y = \cN^\ell \times \{0\}^{k - \ell}$ for some $\ell \le k$.
That $f$ vanishes on $Y \cap \ZZ^\ell = \NN^\ell$ implies that $f(y_1, \dots, y_\ell, 0, \dots, 0)$ is the $0$-polynomial,
and hence $f$ indeed vanishes on $Y$.

Now let $Y$ be $Z$-definable. Using cell decomposition in the version of Lemma~\ref{lem.cell}, we may assume that
$Y$ is of the form
\[
Y = \{(y_1, \dots, y_n) \in \cN^n \mid y_1 \le \ell_1, y_2 \le \ell_2(y_1), \dots, y_n \le \ell_n(y_1, \dots, y_{n-1})\}
\]
where each $\ell_i$ is a $Z$-definable linear function with codomain $\cN \cup \{\infty\}$.

The externally definable subset $Y \cap \ZZ^k$ of $\ZZ^k$ is also internally definable, i.e., there exists a $\emptyset$-definable set $Y' \subseteq \cZ^k$ such
that $Y' \cap \ZZ^k = Y \cap \ZZ^k$. This is true for any definable set $Y$, but the specific form of our set $Y$ ensures that we can moreover choose $Y' \supseteq Y$:
\[
Y' := \{(y_1, \dots, y_n) \in \cN^n \mid y_1 \le \ell'_1, y_2 \le \ell'_2(y_1), \dots, y_n \le \ell'_n(y_1, \dots, y_{n-1})\},
\]
where $\ell'_i = \ell_i$ if $\ell_i(0) \in \NN$, and $\ell'_i = \infty$ otherwise.

Now the lemma follows by applying the $\emptyset$-definable version to $Y'$: Since $f$ vanishes on $Y \cap Z^k \supseteq Y \cap \ZZ^k = Y' \cap \ZZ^k$, it vanishes on $Y' \supseteq Y$.
\end{proof}

\begin{lem}\label{lem.count}
Let $X \subseteq \cZ^n$ be a bounded definable set, say
$X = X_{\btup} := \phi(\cZ, \btup)$ for some $\btup \in \cZ^k$ and some $L$-formula $\phi$.
Then there exists a $\emptyset$-definable set $Y \subseteq \cZ^k$ containing $\btup$
and a polynomial $f \in \QQ[y_1, \dots, y_k]$ such that for every $\btup' \in Y \cap \ZZ^k$,
we have $\#(X_{\btup'} \cap \ZZ^n) = f(\btup')$. (In particular, those $X_{\btup'}$ are finite.)
Moreover, the element $f(\btup) \in \Sym\cQ$ is entirely
determined by the set $X_{\btup}$, i.e., it does not depend on the choices of $\phi$, $\bar b$, $Y$, and $f$.
\end{lem}

\begin{proof}
The claim of the lemma is not affected by applying a definable bijection to $X$, so
by Corollary~\ref{cor.bij-to-cube}, we may assume that $X_{\ytup}$ is a
disjoint union of products of families of the form $[0, \ell_i(\ytup))$ for some $\emptyset$-definable
functions $\ell_i\colon \cZ^k \to \cNi$. Using piecewise linearity of definable functions (Corollary~\ref{lem.lin}), we find a $\emptyset$-definable set $Y \subseteq \cZ^k$
containing $\btup$ such that each $\ell_i$ is linear on $Y$. Moreover, no $\ell_i$ is equal to $\infty$,
since $X_{\btup}$ is bounded (and since definable maps preserve boundedness, by piecewise linearity).
Thus when $\btup'$ runs over $Y \cap \ZZ^k$, $\# X_{\btup'}$ is a sum of products of linear functions in $\btup'$
and hence a polynomial.

For the moreover-part, suppose that we have
$X = \phi(\cZ, \btup) = \phi'(\cZ, \btup')$ for some
formulas $\phi, \phi'$ and some
$\btup \in Z^k, \btup' \in Z^{k'}$. Suppose also that
$Y \subseteq Z^k$ and $Y' \subseteq Z^{k'}$ are
corresponding $\emptyset$-definable sets and $f$ and $f'$ are corresponding polynomials, as above.
We have to prove $f(\btup) = f'(\btup')$.

Without loss, $\btup = \btup'$; otherwise, replace both, $\btup$ and $\btup'$, by $\btup\btup'$. Then we can also suppose $Y = Y'$;
otherwise, replace both sets by the intersection. We further shrink $Y$ by imposing $\phi(\cZ, \ytup) = \phi'(\cZ, \ytup)$.
This implies that $f - f'$ vanishes on $Y \cap \ZZ^k$. Now Lemma~\ref{lem.vanish} implies that $f - f'$ vanishes on all of $Y$,
so in particular $f(\btup) = f'(\btup)$.
\end{proof}

\begin{defn}\label{defn.count}
Given a bounded definable set $X$, we call the element $f(\btup) \in \Sym\cQ$ obtained from Lemma~\ref{lem.count}
the \emph{hyper-cardinality of $X$} and we denote it by $\#X$. (Note that if $X$ is $Z$-definable, then $\#X \in \Sym Q \subseteq \Sym \cQ$,
where $Q = Z \otimes_{\ZZ} \QQ$.)
\end{defn}

\begin{lem}
The hyper-cardinality of bounded $Z$-definable sets $X$ and $X'$ has the following properties:
\begin{enumerate}
\item If there exists a definable bijection $X \to X'$, then $\#X = \#X'$.
\item If $X$ and $X'$ are disjoint, then $\#X + \#X' = \#(X \cup X')$
\item $\#(X \times X') = \#X \cdot \#X'$.
\end{enumerate}
In other words, $\#$ induces a semi-ring homomorphism $\Ksb \to \Sym Q$ which we also denote by $\#$.
\end{lem}

\begin{proof}
Choose $\phi$, $\phi'$ and $\btup$ such that $X = X_{\btup} := \phi(\cZ, \btup)$ and $X' = X'_{\btup} := \phi'(\cZ, \btup)$.
Choose $Y$ as in Lemma~\ref{lem.count}, small enough to work for both, $\phi$ and $\phi'$.

To prove (1), shrink $Y$ further such that for any $\ytup \in Y$, there exists a definable bijection
$X_{\ytup} \to X'_{\ytup}$.
Then in Lemma~\ref{lem.count}, we can use the same polynomial $f$ for both
$X_{\ytup}$ and $X'_{\ytup}$ and we get $\# X_{\btup} = f(\btup) = \# X'_{\btup}$.

To prove (2) and (3), note that in Lemma~\ref{lem.count}, as polynomials for $X_{\btup} \cup X'_{\btup}$ and $X_{\btup} \times X'_{\btup}$, we can take the sum and the product of the polynomials for $X_{\btup}$ and $X'_{\btup}$, respectively.
\end{proof}

Using this hyper-cardinality, we can prove Theorem~\ref{thm.groth} in the case of bounded definable sets.

\begin{prop}\label{prop.groth-b}
The map $\#\colon \Ksb \to \Sym Q$ is the inverse of $\Phi\colon \Sym N \to \Ksb$. In particular, its image is $\Sym N$.
\end{prop}
\begin{proof}
It is enough to verify that $\#$ and $\Phi$ are inverses of each other on a set of generators
of the two semirings. $\Sym N$ is generated by $N$, and by Proposition~\ref{prop.bij-to-cube},
$\Ksb$ is generated by (the classes of) the sets $[0, a)$ for $a \in N$.
By definition, $\Phi$ sends $a$ to $[0, a)$, and one easily checks that $\#[0, a) = a$.
\end{proof}

\subsection{More invariants: a bunch of dimensions}
\label{sect.dim}

To obtain a full set of invariants for arbitrary definable sets (i.e., including unbounded ones), it does not suffice to consider
the hyper-cardinality and (a single notion of) dimension. Instead, we need a whole range of notions of dimension.

\begin{exa}\label{exa.int-dim}
Choose $a, a' \in N \setminus \NN$ satisfying $a' > na$ for all $n \in \NN$.
Then there is no definable bijection between $X_1 := \NN \times [0, a)$ and $X_2 := \NN \times [0, a')$. One can distinguish between those sets using a notion of dimension which only considers sufficiently long
intervals as being 1-dimensional, i.e., which in particular associates $1$ to $X_1$ and $2$ to $X_2$.
\end{exa}

The notions of dimensions will be defined using the rank in suitable matroids. This builds on
results from \cite{For.dim}, where conditions on a matroid are given (being ``existential'') which imply that the
notion of dimension it yields behaves well.

As usual, we fix $Z \prec \cZ$; we apply all results of \cite{For.dim} to the language $L(Z)$.
(Since our language $L$ does not contain $Z$, we slightly adapt the terminology.)

First, we recall some standard definitions.

\begin{defn}
\begin{enumerate}
 \item
A \emph{matroid} is a map $\cl\colon \pow(\cZ) \to \pow(\cZ)$ with the following properties:
$A \subseteq \cl(A)$; if $A \subseteq B$ then $\cl(A) \subseteq \cl(B)$; $\cl(\cl(A)) = \cl(A)$;
$b' \in \cl(A \cup \{b\}) \setminus \cl(A)$ implies
$b \in \cl(A \cup \{b'\})$.
 \item The \emph{rank} $\rk^{\cl}(\atup/B)$ (where $\cl$ is a matroid and $\atup, B \subseteq \cZ$)
   is the length of the shortest tuple $\atup_0 \subseteq \atup$ such that $\cl(\atup_0 \cup B) = \cl(\atup \cup B)$.
\end{enumerate}
\end{defn}

Next, we recall some definitions from \cite{For.dim}.

\begin{defn}\label{defn.matr2}
\begin{enumerate}
\item
A matroid is $Z$-\emph{definable} \cite[Definition~3.15]{For.dim}, if there exists a set $\Psi$ of $L(Z)$-formulas $\psi(x; \ytup)$
(where in different formulas, the tuples $\ytup$ can have different lengths) such that
for every $B \subseteq \cZ$,
\[
\cl(B) = \bigcup_{\psi \in \Psi, \btup \subset B}\psi(\cZ; \btup)
\]
\item
A matroid is \emph{$Z$-existential} \cite[Definition~3.25]{For.dim} if (a) it is $Z$-definable,
(b) $\cl(Z) \ne \cZ$, and (c) for every $a \in \cZ$ and $B, C \subseteq \cZ$ with $a \notin \cl(B)$, there exists $a' \in \cZ$
with $\tp(a'/B)=\tp(a/B)$ such that $a' \notin \cl(B \cup C)$.
\end{enumerate}
\end{defn}

In \cite{For.dim}, only ``finitary'' matroids are considered. Since anyway definable implies finitary, we do not introduce this notion.

To prove $Z$-existentiality, we use:

\begin{lem}[{\cite[Lemma 3.23]{For.dim}}]\label{lem.exist}
Suppose that $\cl$ is a matroid satisfying Definition~\ref{defn.matr2} (2) (a) and (b) and suppose that
for every $A \subseteq \cZ$, $\cl(Z \cup A)$ is an elementary substructure of $\cZ$. Then $\cl$ is $Z$-existential.
\end{lem}

(Note that in \cite{For.dim}, ``satisfying existence'' means satisfying Definition~\ref{defn.matr2} (2) (c).)

Given a $Z$-existential matroid, \cite{For.dim} allows us to introduce a notion of dimension. (Recall that by ``$B$ small'', we mean that $\cZ$ is $|B|^+$-saturated.)

\begin{defn}[{\cite[Definition~3.29]{For.dim}}]\label{defn.dim.tp}
Let $\cl$ be a $Z$-existential matroid. Suppose that $X$ is $\wedge$-definable, say with parameters from a small set $B \supseteq Z$.
Then we define the dimension of $X$ by $\dimcl(X) := \max \{\rk^{\cl}(\atup/B) \mid \atup \in X\}$ (and $\dimcl(X) = -\infty$ if $X = \emptyset$).
\end{defn}

By \cite[Remark 3.30]{For.dim}, $\dimcl(X)$ is well-defined, i.e., it only depends on the set $X$ and not on the set $B$ of parameters.

\begin{rem}\label{rem.dim-bas}
It is clear from the definition that for $X \subseteq Y$, we have $\dimcl X \le \dimcl Y$. Moreover, the condition ``$\cl(Z) \ne \cZ$''
(imposed in existentiality) implies $\dimcl \cZ = 1$.
\end{rem}

The following lemma lists some more properties of dimension that are relevant for us. (Many more are given in \cite{For.dim}.)

\begin{lem}\label{lem.dim-bas}
Suppose that $\cl$ is a $Z$-existential matroid and that $f\colon X \to Y$ is a $\btup$-definable map (for some $\btup \subset \cZ$).
\begin{enumerate}
\item\label{lem.dim-bas.surj}
  If $X' \subseteq X$ is $\wedge$-definable and
  each fiber of the restriction $f' := f|_{X'}\colon X' \to Y' := f(X')$ has dimension $n$,
  then $\dimcl X' = \dimcl Y' + n$.
\item\label{lem.dim-bas.def}
  For every $d \in \NN$, the set $\{y \in Y \mid \dimcl f^{-1}(y) \ge d\}$ is ($Z \cup \btup$)-$\wedge$-definable.
\end{enumerate}
\end{lem}

\begin{proof}
(1) is contained in \cite[Lemma 3.44]{For.dim}.

(2) follows from \cite[Remark 3.31]{For.dim}, namely:
By that remark, the set $\{\atup \mid \rk^{\cl}(\atup/ Z \btup) \ge d\}$ is $(Z \cup \btup)$-$\wedge$-definable,
say by a conjunction of $L(Z \cup \btup)$-formulas $\phi_i(\xtup)$.
Then the set from (2) is defined by the conjunction of $\psi_i(y) := \exists \xtup: (\phi_i(\xtup) \wedge f(\xtup) = y)$.
\end{proof}

In Presburger Arithmetic, $\acl$ satisfies the exchange property (see e.g. \cite[\S3]{Clu.cell} using \cite{BPW.quasiOmin}),
so it is a matroid. Moreover, it is clearly $Z$-definable
in the sense of Definition~\ref{defn.matr2} (for any $Z \prec \cZ$) and $|Z|^+$-saturation of $\cZ$ implies $\acl(Z) \ne \cZ$. Finally, by Lemma~\ref{lem.dcl}, algebraically closed sets are elementary substructures of $\cZ$, so $\acl$ is $Z$-existential by Lemma~\ref{lem.exist}.
The corresponding dimension $\dim^{\acl}$ is the one that we already used in
Section~\ref{sect.bij-to-cube}.

Now consider the following variant: For any set $C \subseteq \cZ$, the ``localisation of $\acl$ at $C$''
$A \mapsto \acl(A \cup C)$ is still a matroid. In general, such a localisation is not $Z$-definable, but it is
if $C$ is $Z$-$\vee$-definable. Moreover, since $\acl(Z \cup A \cup C) \prec \cZ$ for any $A$,
it then also is $Z$-existential, provided that $\acl(Z \cup C) \ne \cZ$. In this way, we obtain a notion of dimension. We apply this as follows.

\begin{lem}\label{lem.isEx}
Let $H$ be a convex subgroup of $Z$, and write $H_{\cZ}$ for the convex closure of $H$ in $\cZ$. Then we have:
\begin{enumerate}
 \item $H_{\cZ}$ is $Z$-$\vee$-definable.
 \item For every $a \in \cZ \setminus H_{\cZ}$, the difference $[0, a) \setminus \acl (Z \cup H_{\cZ})$ is non-empty.
 \item The matroid $A \mapsto \acl(A \cup H_{\cZ})$ is $Z$-existential.
\end{enumerate}
\end{lem}

\begin{proof}
(1) $H_{\cZ}$ is the union of the sets $[-a, a]$ for $a \in H$.

(2) Using that $\acl = \dcl$ (since we have a linear order)
and that definable functions are piecewise linear (Corollary~\ref{lem.lin}), we obtain
that $\acl(Z \cup H_{\cZ})$ is the divisible hull of $Z + H_{\cZ}$ in $\cZ$. Since both $Z$ and $H_{\cZ}$ are relatively divisible in $\cZ$,
we have $\acl(Z \cup H_{\cZ}) = Z + H_{\cZ}$. This set is $Z$-$\vee$-definable: It is the union of the sets $[z - h, z + h]$ for $z \in Z, h \in H \cap N$.
Now consider the $\wedge$-definable set $Y := [0, a) \setminus (Z + H_{\cZ})$;
our goal is to prove $Y \ne \emptyset$. By saturation of $\cZ$, it suffices to show finite satisfiability.
This finite satisfiability follows from looking at 
the quotient $\cZ/H_{\cZ}$:
Each $[z - h, z + h]$ gets send to a singleton in this quotient, whereas the image of $[0, a)$ is infinite,
since $\cZ/H_{\cZ}$ is divisible and $a \notin H_{\cZ}$.

(3) By (1) and $|Z|^+$-saturatedness of $\cZ$, an $a$ as in (2) exists. Thus $\acl(Z \cup H_{\cZ}) \ne \cZ$
and hence the arguments from above the lemma apply.
\end{proof}

\begin{defn}\label{defn.dim}
Let $\csg$ be the set of all convex subgroups of $Z$ except $\{0\}$ (but including $Z$ itself). For $H \in \csg$, let $H_{\cZ}$ be the
convex closure of $H$ in $\cZ$ and define
$\dim_H$ be the dimension function obtained from $H_{\cZ}$ as described right above,
i.e., by applying Definition~\ref{defn.dim.tp} to the matroid $A \mapsto \acl(A \cup H_{\cZ})$.
For a $Z$-$\wedge$-definable set $X$, we define
$\dim_*(X) := (\dim_H(X))_H \in \dimspace$.
On $\dimspace$, we consider the natural partial order:
$(d_H)_H \le (d'_H)_H$ iff $d_H \le d'_H$ for all $H \in \csg$.
\end{defn}

As examples, note that $\dim_\ZZ = \dim^{\acl}$ counts ``dimensions of infinite cardinality'',
whereas $\dim_Z$ counts ``dimensions of unboundedness''; in particular, $\dim_Z(X) = 0$ iff $X$ is bounded (cf.\ Remark~\ref{rem.bd}).
Also note that for $H, H' \in \csg$ with $H \subseteq H'$, we have $\dim_H \ge \dim_{H'}$. To get used to this definition, we
prove:

\begin{lem}\label{lem.dim.int}
For $H \in \csg$ and $a \in N$, the dimension $\dim_H [0, a)$ is $0$ if $a \in H$ and $1$ otherwise.
\end{lem}

\begin{proof}
The case $a \in H$ follows from $[0, a) \subseteq H_{\cZ}$.
If $a \notin H$, then Lemma~\ref{lem.isEx} (2) yields a $b \in [0, a) \setminus \acl(Z \cup H_{\cZ})$;
such a $b$ witnesses the dimension being $1$.
\end{proof}

\begin{notn}
Given a type $p \in S_n(Z)$, we write $\dimcl(p) := \dimcl(p(\cZ))$, and similarly $\dim_*(p) := \dim_*(p(\cZ))$.
\end{notn}

The tuple $\dim_*$ of dimensions is a good invariant when applied to complete types, but for definable sets, it
misses out some information:

\begin{exa}
Choose $a \in N \setminus \NN$ and consider
the definable sets $X_1 := \cN \times [0, a)$ and $X_2 := (\cN \times \{0\}) \cup [0,a)^2$.
Then $\dim_\ZZ(X_1) = \dim_\ZZ(X_2) = 2$ and $\dim_Z(X_1) = \dim_Z(X_2) = 1$, but there is no definable
bijection between $X_1$ and $X_2$, since only $X_1$ contains elements $\btup$ which satisfy both,
$\dim_\ZZ(\tp(\btup/Z)) = 2$ and $\dim_Z(\tp(\btup/Z)) = 1$ simultaneously
(namely, $\btup = (b_1, a-1)$ for $b_1$ outside of the convex closure of $Z$.)
\end{exa}

This example motivates the following, rather technical improved notion of dimension.

\begin{defn}\label{defn.mdim}
For a $Z$-definable set $X \subseteq \cZ^n$,
define the \emph{multidimension} $\mdim(X)$ to be the set of
all maximal elements of the set
\[
M_X := \{\dim_*(p) \mid p \in S_n(Z), p(\cZ) \subseteq X\} \subseteq \dimspace
\]
(i.e., $d \in \mdim(X)$ iff $d \in M_X$ and there is no $d' \in M_X$ with $d' > d$).
\end{defn}

Before we prove that this definition is reasonable (in particular that the above set $M_X$ does have
maximal elements at all), we note that the multidimension is preserved under definable bijections.

\begin{lem}
If $X$ and $Y$ are $Z$-definable and there exists a definable bijection $X \to Y$,
then $\mdim(X) = \mdim(Y)$.
\end{lem}
\begin{proof}
This follows from Lemma~\refi{lem.dim-bas}{surj}, applied separately to each complete type in $X$ (and its image in $Y$).
\end{proof}

To understand multidimensions, we use that any $X$ is in definable bijection to a finite union of cuboids.

\begin{lem}\label{lem.mdim-ex}
For any non-empty $Z$-definable set $X$, $\mdim(X)$ is a finite, non-empty set.
Moreover, for any element $d'$ of the set $M_X$ from Definition~\ref{defn.mdim}, there exists
an element $d \in \mdim(X)$ with $d' \le d$.

More specifically, if $X = [0, a_1) \times \dots \times [0, a_n)$ for some
$a_1, \dots, a_n \in \Ni \setminus \{0\}$, then
$\mdim(X) = \{\dim_*(X)\} = \{(d_H)_H\}$, where $d_H := \#\{i \le n \mid a_i \notin H\}$.
\end{lem}

\begin{proof}
First consider the case that $X$ is a cuboid, $X = [0, a_1) \times \dots \times [0, a_n)$. Let $M_X$ be as in Definition~\ref{defn.mdim} and set
$d := (d_H)_H \in \dimspace$ for $d_H$ as in the lemma.
We prove two claims for the cuboid $X$: (a) $d \in M_X$, and (b) $d' \le d$ for every $d' \in M_X$. Since any $Z$-definable set is in definable bijection to a finite union of cuboids, the full lemma will follow after proving these claims.

We start by proving (b). To this end, consider
an arbitrary $\btup = (b_1, \dots, b_n) \in X$; we have to check that
for every $H \in \csg$, we have $\dim_H(\tp(\btup/Z)) \le d_H$.
For any $i$ with $a_i \in H$, we have $b_i \in H_{\cZ}$ and hence
$\dim_H \tp(b_i/Z) = 0$. Since for all other $i$, we have $\dim_H \tp(b_i/Z) \le 1$,
the claim follows (using Lemma~\refi{lem.dim-bas}{surj}).

To prove (a), we need to find $\btup \in X$ with $\dim_*(\tp(\btup/Z)) = d$.
We assume $a_1 \ge \dots \ge a_n$, and we proceed by induction on $n$, i.e., we assume
that we already have a tuple $\btup' \in X' := [0, a_1) \times \dots \times [0, a_{n-1})$ with the right $\dim_*(\tp(\btup'/Z))$.

Let $H_0$ be the largest convex subgroup of $Z$ not containing $a_n$.
(This exists: it consists of all those $z \in Z$ for which $a_n$ does not lie in the convex closure of the group generated by $z$.)
Then $\dim_{H_0} [0, a_n) = 1$ by Lemma~\ref{lem.dim.int}.
By the definition of dimension (and since we are allowed to choose our parameter set to be $Z \cup \btup'$),
there exists a $b_n \in [0, a_n)$ such that $\rk(b_n/Z\btup') = 1$. Thus, for $\btup := \btup' b_n \in X$,
we have $\dim_{H_0}(\tp(\btup/Z)) = \rk(\btup/Z) =\rk(b_n/Z\btup') + \rk(\btup'/Z) = \dim_{H_0}(\tp(\btup'/Z)) + 1 = d_{H_0}$,
where the rank is the one corresponding to the matroid obtained from $H_0$.

Since by (b), $\dim_*(\bar b) \le d$,
it remains to prove that $\dim_H (\tp(\btup/Z)) \ge d_H$ for every $H \in \csg$.
If $H \supseteq H_0$, then $d_H$ is the same for $X$ and for $X'$, so $\dim_H (\tp(\btup/Z)) \ge \dim_H (\tp(\btup'/Z)) = d_H$.
If, on the other hand, $H \subseteq H_0$, then $\dim_H (\tp(\btup/Z)) \ge \dim_{H_0} (\tp(\btup/Z)) = d_{H_0}$, and this is equal to $d_H$ by our assumption that $a_n \le a_i$ for all $i$.
\end{proof}

\begin{rem}\label{rem.bd}
In particular, using this (and that definable bijections preserve boundedness), one obtains that
a $Z$-definable set $X$ is bounded iff $\dim_Z(X) = 0$.
\end{rem}

\begin{lem}\label{lem.dim-union}
If $X = X_1 \cup X_2$ are $Z$-definable sets, then $\mdim(X)$ is equal to the maximal elements
of the (finite) set $\mdim(X_1) \cup \mdim(X_2)$. In particular, if $X$ is (in definable bijection to)
a disjoint union of cuboids $C_1, \dots, C_\ell$, then
$\mdim(X)$ is equal to the maximal elements of the set $\{\dim_* C_1, \dots, \dim_* C_\ell\}$.
\end{lem}
\begin{proof}
This follows from the definition of $\mdim$ and Lemma~\ref{lem.mdim-ex}.
\end{proof}

We naturally extend the partial order on $\dimspace$ to multidimensions:

\begin{defn}
Given $D, D' \subseteq \dimspace$, we write $D \le D'$ iff for every $d \in D$, there exists a $d' \in D'$ with $d \le d'$.
\end{defn}

\begin{rem}\label{rem.subset}
From Lemma~\ref{lem.dim-union}, we deduce: For $Z$-definable sets $X \subseteq X'$, we have $\mdim X \le \mdim X'$.
\end{rem}

\subsection{Purely unbounded sets}

With multidimension and hyper-cardinality together, we now have all the ingredients needed to distinguish
any definable sets that are not in definable bijection. However, the most naive approach -- using hyper-cardinality for bounded sets and
multidimension for unbounded ones -- does not work, due to some definable sets being a union of a bounded and an unbounded one in an
essential way.

\begin{exa}
Choose $a \in N \setminus \NN$ and consider
the sets $X_1 = \cN \dcup [0,a)^2$ and $X_2 = \cN \dcup [0,a)^2 \dcup [0,a)^2$. They are both unbounded, have the same multidimension, but
they are not in definable bijection, since, as we shall see later (Corollary~\ref{cor.cancel}) this would imply that
also $\cN$ and $\cN \dcup [0,a)^2$ are in definable bijection. (But those sets have different multidimension.)
\end{exa}

The solution is to decompose each definable set into a bounded and a ``purely unbounded'' piece.

\begin{defn}\label{defn.purely-u}
We call a definable set $X$ \emph{purely unbounded} if
every $(d_H)_H \in \mdim(X)$ satisfies
$d_Z \ge 1$.
\end{defn}

Note that a cuboid is always either bounded or purely unbounded.

In the following, $\Phi$ is the map $\Sym \Ni \to \Ks$ from Theorem~\ref{thm.groth}.
We use it to transfer some definitions from definable sets to elements of $\Sym \Ni$.

\begin{defn}
Suppose that $\Phi(a) = [X]$ for some $a \in \Sym\Ni$ and some $Z$-definable set $X$.
We define the \emph{multidimension} of $a$ to be $\mdim (a) := \mdim X$. If $\mdim (a)$ is a singleton,
we denote the unique element of that set by $\dim_* (a)$.
We say that $a$ is \emph{purely unbounded} iff $X$ is purely unbounded.
\end{defn}

Those notions are well-defined since $\mdim X$ depends on $X$ only up to definable bijection.

\begin{rem}\label{rem.dim-union}
We can reformulate Lemma~\ref{lem.dim-union} in terms of $\Sym \Ni$: For $a, a' \in \Sym \Ni$, $\mdim(a + a')$ is the
set of maximal elements of the union $\mdim(a) \cup \mdim(a')$ and if $a$ is a product of elements of $\Ni$, then
$\mdim (a)$ is the singleton described by Lemma~\ref{lem.mdim-ex}. In particular,
if an arbitrary $a \in \Sym \Ni$ is given as a sum $a = a_1 + \dots + a_\ell$, where each $a_i$ is a product of elements of $\Ni$, then
$\mdim (a)$ is the set of maximal elements of $\{\dim_*(a_1), \dots, \dim_*(a_\ell)\}$.
\end{rem}

\begin{rem}\label{rem.pub}
We would like to say that an element $a \in \Sym \Ni$ is purely unbounded iff it can be written as a sum
$a = a_1 + \dots + a_\ell$ of products of elements of $\Ni$ in such a way that each $a_i$ has $\infty$ as a factor.
The implication ``$\Leftarrow$'' is clear. The other direction follows from Theorem~\ref{thm.groth}:
If $\Phi(a) = [X]$ for $X$ purely unbounded, then without loss, $X$ is a disjoint union of unbounded cuboids and
hence $\Phi\1([X])$ can be written as desired.
However, as long as we are still working on the proof of Theorem~\ref{thm.groth}, we cannot yet use this implication.
\end{rem}

\begin{lem}\label{lem.unbound-eat}
Let $a$ and $a'$ be two elements of $\Sym \Ni$. Suppose that
$a'$ is purely unbounded and that $\mdim(a) \le \mdim(a')$. Then $a'$ eats $a$, i.e., $a' + a = a'$ (see Definition~\ref{defn.eats}).
\end{lem}

\begin{proof}
We first prove the ``monomial case'', i.e., we assume that $a = a_1\cdots a_n$
and $a' = a'_1 \cdots a'_{n'}$
for some $a_1, \dots, a_n, a'_1, \dots, a'_{n'} \in \Ni$.
We may assume $n = n'$ (otherwise, insert additional factors $1$) and
$a_1 \le \dots \le a_n$, $a'_1 \le \dots \le a'_n$.
Set  $(d_H)_H := \dim_* (a)$ and $(d'_H)_H := \dim_* (a')$.

To prove that $a'$ eats $a$, we use Lemma~\ref{lem.eat}. Since $a'$ is purely unbounded,
we have $a'_n = \infty$, so the only missing ingredient to Lemma~\ref{lem.eat} is
that for each $i \le n$, there exists a $k \in \NN$ such that $ka'_i \ge a_i$.
Suppose this fails for $i$. We denote by $H \in \csg$ the smallest convex subgroup of $Z$ containing $a'_i$.
Then $a'_i \in H$ (and also $a'_j \in H$ for all $j < i$) but $a_i \notin H$ (and $a_j \notin H$ for $j > i$).
By Lemma~\ref{lem.mdim-ex}, this implies $d'_H \le n - i$ but $d_H > n - i$.
This contradicts $(d_H)_H \le (d'_H)_H$.

Now consider the general case. We write $a = b_1 + \dots + b_m$ and $a' = b'_1 + \dots + b'_{m'}$,
where $b_i$, $b'_j$ are monomials as in the above case.
It suffices to find, for each $i \le m$, a $j \le m'$ such that $b'_j$ eats $b_i$, so let $i$ be given.
Since $\mdim(b_i) \le \mdim(a) \le \mdim(a')$, there exists a $d' \in \mdim(a')$ such that $\dim_*(b_i) \le d'$.
By Remark~\ref{rem.dim-union}, $d' = \dim_*(b'_j)$ for some $j \le m'$.
Now $b'_j$ eats $b_i$ by the monomial case (and using that because $\dim_*(b'_j) \in \mdim(a')$, it has $\infty$ as a factor).
\end{proof}

\subsection{Proof of Theorem~\ref{thm.groth} (combining hyper-cardinality and dimensions)}

\begin{proof}[Proof of Theorem~\ref{thm.groth}]
The image of the map $\Phi\colon \Sym \Ni \to \Ks$ consists exactly of all finite disjoint unions of cuboids, so by Proposition~\ref{prop.bij-to-cube},
it is surjective.
By Proposition~\ref{prop.groth-b}, $\Phi$ restricts to
an isomorphism of sub-semirings $\Sym N \to \Ksb$.
It remains to check that $\Phi$ is injective.

Suppose that $\Phi(a) = \Phi(a'') = [X]$ for some definable set $X$.
Instead of proving directly $a = a''$ we will find an intermediate element $a' \in \Sym \Ni$
and prove $a = a'$ and $a' = a''$; this $a'$ is defined as follows.

We can write $a$ as a sum $a = a_{u} + a_{b}$, where $a_{u}$ is purely unbounded and $a_{b} \in \Sym N$.
Indeed, if we write $a$ as a sum of products of elements of $\Ni$, we can take $a_{u}$ to be the sum
of those summands with factor $\infty$ and $a_{b}$ to be the sum of those without.
Set $[X_{b}] := \Phi(a_{b})$ and $[X_{u}] := \Phi(a_{u})$. Since $[X_{b}] + [X_{u}] = [X]$, we may assume that $X = X_{b} \dcup X_{u}$.
Do the same with $a''$ to obtain a second partition $X = X''_{b} \dcup X''_{u}$ of $X$
and set $X'_{b} := X_{b} \cap X''_{b}$ and $X'_{u} := X_{u} \cup X''_{u}$; then we also have
$X = X'_{b} \dcup X'_{u}$.

Choose any preimages
$a'_{b} \in \Phi\1([X'_{b}])$ and $a'_{u} \in \Phi\1([X'_{u}])$ and set $a' := a'_{b} + a'_{u}$.
We will now prove $a = a'$; the proof of $a' = a''$ works analogously.

Set $X_m := X_{b} \cap X'_{u}$; then we have: $X = X_{u} \dcup X_m \dcup X'_{b}$, $X'_{u} = X_{u} \dcup X_m$, $X_{b} = X'_{b} \dcup X_m$.

By definition of $a_{b}, a''_{b}, a_{u}, a''_{u}$ and by Lemma\ref{lem.mdim-ex}, $X_{b}, X''_{b}$ are bounded and $X_{u}, X''_{u}$ are purely unbounded.
Lemma~\ref{lem.dim-union} implies that $X'_{u}$ is purely unbounded, too.
However, $X'_{u} = X_{u} \dcup X_m$, where $X_m$ is bounded (since it is a subset of $X_{b}$),
so using Lemma~\ref{lem.dim-union}, we deduce $\mdim(X_m) \le \mdim(X'_{u}) = \mdim(X_{u})$.

Set $a_m := \#X_m$. By Proposition~\ref{prop.groth-b}, we have
$\Phi(a_m) = [X_m]$ and $a'_{b} + a_m = a_{b}$ (since $\Phi(a'_{b}) + \Phi(a_m) = \Phi(a_{b})$ and $a_m, a_{b}, a'_{b} \in \Sym N$).
Now we apply Lemma~\ref{lem.unbound-eat} three times: $\mdim(X'_{u}) = \mdim(X_{u})$ implies that $a'_{u}$ eats $a_{u}$ and vice versa and using $\mdim(X_m) \le \mdim(X_{u})$,
we obtain that $a_{u}$ also eats $a_m$. This allows us to finish the computation:
\[
a = a_{b} + a_{u} = a'_{b} + a_m + a_{u} = a'_{b} + a_{u} = a'_{b} + a_{u} + a'_{u} = a'_{b} + a'_{u} = a'
.
\]
\end{proof}

\section{Consequences}
\label{sect.cor}

\subsection{Consequences for definable sets}

By now, we have an almost complete algorithm to find out whether two given definable sets are in definable bijection (assuming that
$Z$ is given in a suitable ``algorithmic way''; we leave it to the reader to make this precise), namely:
By following the proof of Proposition~\ref{prop.bij-to-cube}, one can turn each definable set into a finite disjoint union of cuboids
and hence write its preimage under $\Phi\colon \Sym \Ni \to \Ks$ as an expression in the generators $\Ni$ of $\Sym \Ni$. The only missing ingredient
is a way to find out whether two such expressions are equal. This is what Proposition~\ref{prop.eq.test} provides.

To state it, it is handy to extend multidimension inequalities to $\Sym Z$.

\begin{defn}
Given $a \in \Sym Z$ and $a' \in \Sym \Ni$, we write $\mdim (a) \le \mdim (a')$ if there exist $b, b'\in \Sym N$
with $a = b - b'$ and with $\mdim(b), \mdim(b') \le \mdim (a')$.
\end{defn}

Note that we did not define ``$\mdim (a)$'' itself. Probably this could also be defined, but we do not need it.
Note also that in the case that $a \in \Sym N$, this new definition of $\mdim (a) \le \mdim (a')$ agrees with
the previous definition.

\private{This last claim needs a short computation. Checked. -Immi}

\begin{prop}\label{prop.eq.test}
Suppose that $a = a_{b} + a_{u}$ and $a' = a'_{b} + a'_{u}$ are elements of $\Sym\Ni$, where
$a_{b}, a'_{b} \in \Sym N$ and $a_{u}, a'_{u}$ are purely unbounded.
Then $a = a'$ if and only if the following two conditions hold:
\begin{enumerate}
 \item $\mdim (a_{u}) = \mdim (a'_{u})$
 \item $\mdim(a_{b} - a'_{b}) \le \mdim (a_{u})$.
\end{enumerate}
Moreover (1) implies $a_{u} = a'_{u}$.
\end{prop}

\begin{proof}
``$\Leftarrow$'': (This is essentially already contained in the proof of Theorem~\ref{thm.groth}.)
By Lemma~\ref{lem.unbound-eat}, $\mdim (a_{u}) = \mdim(a'_{u})$ implies
$a_{u} = a_{u} + a'_{u} = a'_{u}$. (This also proves the moreover part.) By (2), we have $a_{b} - a'_{b} = a_m - a'_m$ for some $a_m, a'_m \in \Sym N$
satisfying $\mdim (a_m), \mdim(a'_m) \le \mdim (a_{u})$. By Lemma~\ref{lem.unbound-eat} again, $a_m$ and $a'_m$ are eaten by $a_{u}$, so we have
$a_{u} + a_{b} = a_{u} + a'_{b}$, from which $a = a'$ follows.

``$\Rightarrow$'' (1) By Remark~\ref{rem.dim-union} and using that no element of $\mdim (a_{b})$ is bigger than any element of $\mdim(a_{u})$, $\mdim(a)$ determines $\mdim(a_{u})$.

(2) Denote by $X, X_{u}, X_{b}, X'_{u}, X'_{b}$ definable sets corresponding to $a, a_{u}, a_{b}, a'_{u}. a'_{b}$, chosen in such a way that $X = X_{u} \dcup X_{b} = X'_{u} \dcup X'_{b}$.
Set $X_{m} := X'_{u} \cap X_{b}$ and $a_{m} := \Phi\1([X_{m}])$. Since $X_{m} \subseteq X'_{u}$, we have $\mdim(a_m) \le \mdim(a'_{u})$
(see Remark~\ref{rem.subset}). Similarly, $\mdim(a'_m) \le \mdim(a_{u})$, where $a'_m := \Phi\1([X'_{m}])$ for $X'_{m} := X_{u} \cap X'_{b}$.
For $a_{bb} := \Phi\1([X_{b} \cap X'_{b}])$, we have $a_{b} - a'_{b} = (a_{bb} + a_m) - (a_{bb} + a'_m) = a_m - a'_m$, so (2) follows.
\end{proof}

From this, we can deduce a complete classification of who eats whom.

\begin{cor}\label{cor.eat}
For $a, a' \in \Sym \Ni$, we have $a' + a = a'$ iff $\mdim (a) \le \mdim (a'_{u})$,
where $a'_{u}$ is obtained from $a'$ as in Proposition~\ref{prop.eq.test}.
\end{cor}

\begin{proof}
Let $a = a_{u} + a_{b}$ and $a' = a'_{u} + a'_{b}$ be as in Proposition~\ref{prop.eq.test}. By the proposition, $a' + a = a'$ is equivalent to
(1) $\mdim(a'_{u} + a_{u}) = \mdim(a'_{u})$ and (2) $\mdim((a'_{b} + a_{b}) - a'_{b}) \le \mdim (a'_{u})$.
By Remark~\ref{rem.dim-union}, (1) is equivalent to $\mdim(a_{u}) \le \mdim(a'_{u})$, and (2) simplifies to $\mdim(a_{b}) \le \mdim(a'_{u})$.
Using Remark~\ref{rem.dim-union} once more, the conjunction of (1) and (2) is equivalent to  $\mdim(a_{u} + a_{b}) \le \mdim(a'_{u})$, which is what we claimed.
\end{proof}

We can now infer some general corollaries about definable sets in $\cZ$. First, here is a Cantor--Schr\"oder--Bernstein like result.

\begin{cor}\label{cor.inin}
If $X$ and $X'$ are definable sets and there exist definable injections $X \to X'$ and $X'\to X$, then $[X] = [X']$.
\end{cor}

\begin{proof}
In terms of elements of $\Sym \Ni$, we need to prove: If $a + b + b' = a$ (for $a, b, b' \in \Sym \Ni$) then $a + b = a$.
This follows using Corollary~\ref{cor.eat}: $\mdim (b + b') \le \mdim (a_{u})$ implies $\mdim (b) \le \mdim (a_{u})$.
\end{proof}

The next corollary states that bounded sets can be additively cancelled.

\begin{cor}\label{cor.cancel}
Suppose that $X$, $X'$ and $Y$ are definable sets, where $Y$ is bounded. Then $[X \dcup Y] = [X' \dcup Y]$ implies $[X] = [X']$.
\end{cor}

\begin{proof}
Set $a := \Phi\1([X]), a' := \Phi\1([X']), b := \Phi\1([Y])$ and let $a = a_{u} + a_{b}$, $a' = a'_{u} + a'_{b}$ be decompositions as in Proposition~\ref{prop.eq.test}.
Then applying the proposition to $a_{u} + (a_{b} + b) = a'_{u} + (a'_{b} + b)$ yields $\mdim (a_{u}) = \mdim (a'_{u})$
and $a_{b} - a'_{b} = (a_{b} + b) - (a'_{b} + b) \le \mdim (a_{u})$ and hence (using the proposition again) $a = a'$.
\end{proof}

It might seem plausible that non-empty bounded sets also cancel multiplicatively. This is clearly true in $\ZZ$ (using the classification of definable sets),
and also in $Z = \QQ \times \ZZ$ (ordered lexicographically), it is true; we leave this to the reader as an exercise. However,
it is false in general, as the following example shows. (The idea behind the example is that addition of multidimension is not cancellative.)

\private{I did the exercise. -Immi}

\begin{exa}
Suppose that there exist $a, a' \in N \setminus \NN$ satisfying $a' > na$ for all $n \in \NN$.
Consider $X = \cN \times [0,a)^4 \dcup \cN \times [0,a')^2$, $X' = X \dcup \cN \times [0,a') \times [0,a)^2$ and $Y = [0,a)^2 \dcup [0,a')$.
Then $[X \times Y] = [X' \times Y]$, but $[X] \ne [X']$ (as one can verify using Proposition~\ref{prop.eq.test}).
\end{exa}

However, multiplicative cancellation at least works if all involved sets are bounded:

\begin{cor}\label{cor.cancel.mult}
Suppose that $X, X', Y$ are definable sets, all three bounded and with $Y \ne \emptyset$. Then $[X \times Y] = [X' \times Y]$ implies $[X] = [X']$.
\end{cor}

\begin{proof}
By Lemma~\ref{lem.SQ}, $\Ksb = \Sym Z$ is an integral domain.
\end{proof}

\subsection{Consequences for definable families}
\label{sect.fam}

Finally, we deduce our main results about definable families.
As usual, we assume that we are given models $\ZZ \prec Z \prec \cZ$, with $\cZ$ being sufficiently saturated.
Recall that for a bounded definable set $X$, $\# X$ denotes its hyper-cardinality (Definition~\ref{defn.count})
and that $\# X$ is just the usual cardinality of $X$ if finite (in particular if $Z = \ZZ$).

\begin{prop}\label{prop.fam-poly}
Suppose that $(X_{\bar y})_{\bar y \in Y}$ is a $Z$-definable family of bounded sets (where $Y$ is also $Z$-definable).
Then there exists a partition of $Y$ into finitely many $Z$-definable sets $Y_i$ and polynomials $f_i \in \Sym Q[\ytup]$
such that for every $\btup \in Y_i$, $\#X_{\btup} = f_i(\btup)$.
\end{prop}

\begin{proof}
This follows from Corollary~\ref{cor.bij-to-cube}, using that definable functions are piecewise linear (Corollary~\ref{lem.lin}),
and using the isomorphism of semirings $\Ksb \to \Sym N$ from Proposition~\ref{prop.groth-b}.
\end{proof}

\begin{thm}\label{thm.pres-fam}
Suppose that $Y \subseteq \cZ^k$ is a $Z$-definable set
and that $X_{\ytup}$ and $X'_{\ytup}$ are two $Z$-definable families, where $\ytup$ runs over $Y$.
Suppose moreover that for each $\btup \in Y$, $X_{\btup}$ and $X'_{\btup}$ are bounded.
Then the following are equivalent:
\begin{enumerate}
 \item For every $\btup \in Y \cap Z^k$, there exists a definable bijection $X_{\btup} \to X'_{\btup}$.
 \item For every $\btup \in Y \cap Z^k$, we have $\# X_{\btup} = \# X'_{\btup}$.
 \item For every $\btup \in Y$, there exists a definable bijection $X_{\btup} \to X'_{\btup}$.
 \item For every $\btup \in Y$, we have $\# X_{\btup} = \# X'_{\btup}$.
 \item\label{5} There exists a definable family of bijections $f_{\ytup}\colon X_{\ytup} \to X'_{\ytup}$.
\end{enumerate}
\end{thm}

Let us for the clarity of (\ref{5}) remind the reader that a collection of sets $X_{\ytup}$ for $\ytup\in Y$ is called a definable family if the set
$\{(\ytup,x)\mid \ytup\in Y,\ x\in X_{\ytup} \}$ is a definable set; a family of functions is called a definable family if the collection of the graphs forms a definable family.


The following example shows that the boundedness assumption in the theorem is indeed necessary.

\begin{exa}\label{ex.unbd-fam}
Set $Z = \ZZ$,
$Y = \cN$, $X_{y} = \cN$ and $X'_{y} = \cN \times [0, y)$.
For every $y \in \NN$, there exists a definable bijection $\NN \to \NN  \times [0, y)$, but the sets $X := \bigcup_{y} X_{y} \times \{y\}$ and
$X' := \bigcup_{y} X'_{y} \times \{y\}$ have different dimension, so in particular there is no definable family of bijections $X_{y} \to X'_{y}$.
\end{exa}

This counter-example only works in the case $Z = \ZZ$; here is another one which works in (certain) bigger models.

\begin{exa}
Choose $a \in N$ in such a way that the smallest convex subgroup of $Z$ containing $a$ is $Z$ itself. (We suppose that such an $a$ exists.)
Set $Y = \cN$, $X_{y} = \cN \times [0, a)$ and $X'_{y} = N \times [0, a + y)$.
Then for every $y \in N$, there exists a definable bijection $X_{y} \to X'_{y}$; this can either be verified by specifying the bijection
explicitly (this is left as an exercise) or using Theorem~\ref{thm.groth} and by noticing that $\infty \cdot a = \infty \cdot (a + y)$
(since $\infty \cdot a$ eats $\infty \cdot y$ by Lemma~\ref{lem.eat}).
On the other hand, $\dim_Z X = 2$ whereas $\dim_Z X' = 3$ (where as before, $X := \bigcup_{y} X_{y} \times \{y\}$ and
$X' := \bigcup_{y} X'_{y} \times \{y\}$), so again no family of definable bijection exists.
\end{exa}

\begin{proof}[Proof of Theorem~\ref{thm.pres-fam}]
(1) and (2) are equivalent by Proposition~\ref{prop.groth-b}, and similarly, (3) and (4) are equivalent.
(5) implies all the other statements, and (3) implies (5) by compactness (or, more precisely, saturatedness of $\cZ$)
and using that by Lemma~\ref{lem.dcl} and Remark~\ref{rem.bij}, a definable bijection between $\btup$-definable sets is itself already $\btup$-definable.
Thus it remains to prove (e.g.) that (2) implies (4).

Let $\btup \in Y$ be given; we want to prove that $\# X_{\btup} = \# X'_{\btup}$.
By Proposition~\ref{prop.fam-poly}, there exists a $Z$-definable subset $\hat Y \subseteq Y$
containing $\btup$
and polynomials $f, f' \in \Sym Q[\bar y]$ such that $\#X_{\btup'} = f(\btup')$ and $\#X'_{\btup'} = f'(\btup')$
for every $\btup' \in \hat Y$.
The assumption (2) implies that $f - f'$ vanishes on $\hat Y \cap Z^{k}$.
Using Lemma~\ref{lem.vanish}, we deduce that $f - f'$ vanishes on all of $\hat Y$, so in particular $\# X_{\btup} = f(\btup) = f'(\btup) = \# X'_{\btup}$.
\end{proof}

\bibliographystyle{amsplain}
\bibliography{references}
\end{document}